\documentclass[a4paper,twoside, 11pt]{amsart}

\usepackage[utf8]{inputenc}
\usepackage[T1]{fontenc}
\usepackage[british]{babel}
\usepackage{lmodern}

\usepackage{amssymb,amsmath,amsthm,amscd,amsfonts}
\usepackage{mathrsfs}
\usepackage{graphicx}
\usepackage{enumerate}
\usepackage{verbatim}
\usepackage[all]{xy}
\usepackage{xcolor}

\selectfont

\theoremstyle{definition}
\newtheorem{thm}{Theorem}[section]
\newtheorem{defn}[thm]{Definition}
\newtheorem{lem}[thm]{Lemma}
\newtheorem{prop}[thm]{Proposition}
\newtheorem{cor}[thm]{Corollary}

\newtheorem{rem}[thm]{Remark}
\numberwithin{equation}{section}

\newcommand{\R}{\mathbb{R}}
\newcommand{\N}{\mathbb{N}}

\newcommand{\C}{\mathbb{C}}

\newcommand{\fL}{\mathcal{L}}
\newcommand{\Lc}{\mathrm{Ly}}

\newcommand{\B}{\mathcal{B}^{\infty}}
\newcommand{\BG}{\mathcal{B}^{\infty}(G)}
\newcommand{\Bc}{\mathcal{B}^{\infty}_c}
\newcommand{\fW}{\mathcal{W}}
\newcommand{\ftL}{\mathcal{\widetilde{L}}}

\newcommand{\om}{\omega}

\newcommand{\sg} {\sigma}

\newcommand{\supp}{\text{supp}}

\newcommand{\pf}{\textnormal{PF}}

\newcommand{\la}{\langle}
\newcommand{\ra}{\rangle}

\renewcommand{\subset}{\subseteq}

\definecolor{darkgreen}{rgb}{0.0, 0.5, 0.0}
\newcounter{timcomments}

\author[Anderson-Sackaney]{Benjamin Anderson-Sackaney}
\address{Department of Mathematics and Statistics, University of Victoria, Victoria, British Columbia, V8W 2Y2, Canada}
\email{bandersonsackaney@uvic.ca}
\thanks{BA-S was supported by NSERC Discovery Grant RGPIN-2026-04876 and NSERC Discovery Launch DGECR-2026-00022.}

\author[de Laat]{Tim de Laat}
\address{Department of Mathematics, Kiel University, Heinrich-Hecht-Platz 6, D-24118 Kiel, Germany}
\email{delaat@math.uni-kiel.de}

\author[Samei]{Ebrahim Samei}
\address{Department of Mathematics and Statistics, University of Saskatchewan, Saskatoon, Saskatchewan, S7N 5E6, Canada}
\email{ebrahim.samei@usask.ca}
\thanks{ES was partially supported by NSERC Discovery Grant RGPIN-2025-04833.}

\author[Wiersma]{Matthew Wiersma}
\address{Department of Mathematics and Statistics, University of Winnipeg, 515 Portage Avenue, Winnipeg, Manitoba, Canada  R3B 2E9}
\email{m.wiersma@uwinnipeg.ca}
\thanks{MW was partially supported by NSERC Discovery Grant RGPIN-2022-03062.}

\title{Furstenberg discretization with moment control and entropy rigidity}

\begin{document}

\begin{abstract}
Let $G$ be a locally compact group, let $\Gamma$ be a lattice in $G$, and let $\eta$ be a spread-out probability measure on $G$. A Furstenberg discretization of $\eta$ to $\Gamma$ is a probability measure $\mu$ on $\Gamma$ such that every $(G,\eta)$-stationary space is also $(\Gamma,\mu)$-stationary. For cocompact lattices $\Gamma$ and under mild regularity conditions on $\eta$, we construct Furstenberg discretizations satisfying finite moment conditions of arbitrary order with respect to arbitrary prescribed subadditive functions. Our approach relies on a rather general study of differential entropy and Lyapunov exponents in the setting of locally compact groups. We relate Lyapunov exponents to spectral radii in weighted $L^{1}$-algebras and prove vanishing results for weights coming from length functions. Additionally, we prove that if $G$ has the property of rapid decay or the Kunze--Stein property, this yields entropy formulas for stationary actions whose Koopman representation is weakly contained in the left-regular representation. As consequences, we obtain preservation results for Poisson boundaries under discretization and characterize Zimmer amenability of stationary actions in terms of weak containment and maximal Furstenberg entropy.
\end{abstract}

\maketitle

\section{Introduction} \label{sec:introduction}


Stationary actions originate in Furstenberg's seminal work \cite{Furstenberg1963a,FurstenbergNCRP} and nowadays constitute a central class of dynamical systems at the interface of ergodic theory, group theory, and probability theory. A key advantage over the purely measure-preserving setting is that, under mild assumptions, stationary probability measures always exist, even in the absence of invariant probability measures. The framework of stationary dynamical systems simultaneously covers measure-preserving actions, where the measure is invariant under the action by every group element, and boundary actions, such as (Furstenberg--)Poisson boundaries, where typically no invariant measure exists, but there is a canonical stationary measure encoding the asymptotic behaviour of the random walk. Since the Haar measure on a non-compact group is infinite, many natural dynamical systems cease to be measure preserving, but they still admit canonical stationary measures, which makes the stationary viewpoint particularly well suited for locally compact groups. Stationary actions have played a substantial role in several historical breakthroughs, including Margulis' superrigidity theorem \cite{Marg1} and the Nevo--Zimmer theorem
\cite{NevoZimmer2002}. For further details on stationary dynamical systems, we refer to the survey \cite{Furs-Glas 1}.

More precisely, let $G$ be a locally compact second countable (lcsc) group with fixed left Haar measure $m$, and let $\eta$ be a spread-out Borel probability measure on $G$. A measurable $G$-space is a standard Borel probability space $(X,\xi)$ together with a measurable action of $G$ on $(X,\xi)$. The action $(G,\eta) \curvearrowright (X,\xi)$ is called \textbf{stationary} if $\eta * \xi = \xi$, meaning that the measure $\xi$ is preserved ``on average'' under the group action; see Section \ref{subsec:stationarydynamicalsystemspoissonboundary} for details. In this case, we refer to $(X,\xi)$
as a $(G,\eta)$-\textbf{stationary space}.

Furstenberg's discretization method originates in his work on random walks on semisimple Lie groups \cite{Furstenberg1973}, in which he related random walks on $G$ to random processes on suitable lattices in $G$. This procedure provides a way to replace a random walk on $G$ by a random walk on a countable subgroup $\Gamma$ (e.g.~a lattice or a countable dense subgroup), while preserving, as far as possible, the stationary and boundary structure of the original system. Concretely, starting from a probability measure $\eta$ on $G$ and a homomorphism $\rho \colon \Gamma \to G$ with lattice or dense image, one constructs a probability measure $\mu$ on $\Gamma$ such that every $(G,\eta)$-stationary space becomes $(\Gamma,\mu)$-stationary via the action twisted by the homomorphism $\rho$. This allows us to study continuous actions and boundary phenomena of $G$ through models on $\Gamma$, and to transfer entropy values, Poisson boundary results, and rigidity phenomena between the continuous and the discrete settings. We now give the explicit definition of a Furstenberg discretization; we only give it for lattices, as that is the setting covered in this article.

\begin{defn}\label{D:Furstenberg discretization}
 Let $\Gamma$ be a lattice in a locally compact second countable group $G$ and $\eta$ a spread-out probability measure on $G$. A non-degenerate probability measure $\mu$ on $\Gamma$ is a \textbf{Furstenberg discretization} of $\eta$ if every $(G,\eta)$-stationary space is $(\Gamma,\mu)$-stationary.
\end{defn}
Equivalently, a non-degenerate probability measure $\mu$ is a Furstenberg discretization of $\eta$ if the Poisson boundary $(\Pi_\eta, \nu_\infty)$ of $(G,\eta)$ is $(\Gamma,\mu)$-stationary.

Furstenberg showed that whenever $G$ is compactly generated and $\eta$ is compactly supported, then $\eta$ admits a Furstenberg discretization \cite{Furstenberg1973} (see also \cite[Proposition VI.4.1]{Marg1}). As mentioned above, the Furstenberg discretization goes back to \cite{Furstenberg1973}; the above definition is a variation on the definition from \cite[Section 1.1]{Bjr-Hartman-Opp 1}. The latter, however, does not consider a lattice $\Gamma$ in $G$, but a homomorphism $\rho \colon \Gamma \to G$ with dense image (see also \cite[Remark 1.2]{Bjr-Hartman-Opp 1}).




For many applications the mere existence of a Furstenberg discretization is not sufficient; one also needs quantitative analytic and asymptotic control on the discretization measure $\mu$. In particular, it is natural to ask whether $\mu$ can be chosen to have finite moments with respect to a prescribed length (or, more generally, a prescribed subadditive function). Such moment conditions play a key role in many arguments on the behaviour of random walks, and they are highly relevant when one transfers analytic or probabilistic arguments between $G$ and $\Gamma$. This phenomenon already appears in classical discretization results. In the Lie group setting, the Furstenberg construction yields, under certain assumptions, discretization measures with finite first and second moments; see \cite[Theorem 2.16]{Creutz} and the references therein. Similarly, in the discretization of Brownian motion and harmonic measures, one finds measures with suitable integrability properties; see \cite[Theorem 1.3]{Alvarez} and the references therein.


This paper pursues two fundamental and closely related objectives. First, it transfers and generalizes the framework of methods developed in \cite{ASdLSW1} from the special case of discrete groups to the general setting of locally compact groups. This yields a characterization of stationary actions of lattices in groups with the rapid decay property or the Kunze-Stein property (see Section \ref{sec:preliminaries}), under the assumption that the associated Koopman representation is weakly contained in the left-regular representation. Second, we prove a Furstenberg discretization result with strong moment control for cocompact lattices. Under mild regularity assumptions on the original random walk, we construct discretizations with finite moments of arbitrary order with respect to every prescribed subadditive function on $G$ (when restricted to $\Gamma$). In this second part, we heavily rely on the new collection of methods developed in the first part.

As an important first result (see Theorem \ref{T:Square integrable-inverse Harich Chandra function-the general case}), we show that under a sufficiently fast decay assumption of the the Harish-Chandra function $\Xi_{X}$ associated with a stationary action $(G,\eta) \curvearrowright (X,\xi)$, the Furstenberg entropy of $(X,\xi)$ must be equal to the Furstenberg entropy of the $\eta$-random walk on its Poisson boundary. This result is a generalization of \cite[Theorem 6.2]{ASdLSW1}, which stated a weaker version of this for discrete groups with rapid decay. 

The aforementioned theorem enables us to characterize stationary actions of groups with property RD or the Kunze–Stein property when the Koopman representation is weakly contained in the left-regular representation. In doing so, we also generalize \cite[Theorem 6.4]{ASdLSW1}, which was established for countable discrete groups with rapid decay.

\begin{thm}\label{T:charc amen action-RD groups-in introduction} (Theorem \ref{T:charc amen action-RD groups})
Let $G$ be a second countable locally compact group, and let $\eta$ be a probability measure on $G$ with $\eta\in \BG$ such that $\eta$ has finite $n^{th}$-differential entropy, for all $n\in \N$.  Suppose that one of the following holds:\\
$(1)$ $(G,\fL)$ has property RD and $\eta$ has a finite $\log(1+\fL)$ moment.\\
$(2)$ $G$ has the Kunze-Stein Property.\\
Then 
for every $(G,\eta)$-stationary space $(X,\xi)$, the following are equivalent:
\begin{enumerate}[(i)]
    \item The space $(X,\xi)$ is a measure-preserving extension of the Poisson boundary of $(G,\eta)$.
    \item The action $G \curvearrowright (X,\xi)$ is amenable (in the sense of Zimmer).
    \item The representation $\pi_{X}$ is weakly contained in $\lambda_G$.
    \item $h_\eta(X,\xi)=h_\eta(\Pi_{\eta},\nu_{\infty})$.
\end{enumerate}
\end{thm}

We also establish similar results on lattices in groups with property RD or the Kunze–Stein property, with respect to Poisson boundaries arising from Furstenberg discretizations. In this setting, the Furstenberg discretization measure provides a natural bridge between the boundary theory of the ambient group and that of its lattices.

\begin{thm}(Theorem \ref{T:Square integrable-inverse Harich Chandra function-Discretization})
Let $\Gamma$ be a lattice subgroup in a locally compact second countable group $G$, let $\eta \in \B_c(G)$, and let $\mu$ be a Furstenberg discretization of $\eta$. Suppose that one of the following two conditions is satisfied:
\begin{enumerate}[(i)]
    \item the group $G$ has property RD with respect to a length function $\fL$ and $\mu$ has both finite entropy and finite $\log (1+\fL)$-moment;
    \item the group $G$ has the Kunze-Stein property and $\mu$ has finite entropy. 
\end{enumerate}
If $(\Pi_\eta, \nu_\infty)$ is the Poisson boundary of $(G,\eta)$, then
\begin{align}\label{Eq:maximal entropy-lattice}
h_\mu(\Pi_\eta, \nu_\infty)=h(\Gamma,\mu).
\end{align}
In other words, $(\Pi_\eta, \nu_\infty)$ is the measure-preserving extension of the Poisson boundary of $(\Gamma,\mu)$ as a $\Gamma$-space. Moreover, if $\Gamma$ is cocompact, then $(\Pi_\eta, \nu_\infty)$ is the Poisson boundary of $(\Gamma,\mu)$ as a $\Gamma$-space.
\end{thm}


In the setting where the ambient group is a connected semisimple Lie group and the Furstenberg discretization on a lattice is induced by the unique $K$-invariant probability measure on the Poission boudary, we are able to give a complete characterization of stationary spaces with Koopman representation weakly contained in the left-regular representation. This provides a class of examples not covered by Theorem \ref{T:charc amen action-RD groups-in introduction} (i.e. Theorem \ref{T:charc amen action-RD groups}), since the corresponding lattices may fail to possess either property RD or the Kunze–Stein property.

\begin{thm}(Theorem \ref{T:Square integrable-inverse Harich Chandra function-lattice in Lie groups}) 
Let $G$ be a connected semisimple Lie group, let $\Gamma$ be a lattice in $G$,  and let $\eta\in \BG$ be such that the Poisson measure on $G/P$ is the unique $K$-invariant probability measure $\nu_\infty^K$. Suppose that $\mu$ is a Furstenberg discretization of $\eta$ to $\Gamma$ such that $\mu$ has both finite entropy and finite logarithmic moment with respect to the Riemannian metric $\fL$ on $G$. 
Then 
for every $(\Gamma,\mu)$-stationary space $(X,\xi)$, the following are equivalent:
\begin{enumerate}[(i)]
    \item The space $(X,\xi)$ is a measure-preserving extension of the Poisson boundary of $(\Gamma, \mu)$.
    \item The action $\Gamma \curvearrowright (X,\xi)$ is amenable (in the sense of Zimmer).
    \item The representation $\pi_{X}$ is weakly contained in $\lambda_\Gamma$.
    \item $h_\mu(X,\xi)=h_\mu(G/P,\nu^K_{\infty})$.
\end{enumerate}
\end{thm}


Finally, we show that Furstenberg discretizations of cocompact lattices can be chosen to satisfy strong moment bounds. Specifically, under mild regularity hypotheses on the original random walk, we construct discretization measures on the lattice satisfying finite moments of arbitrary order for all subadditive functions on the ambient group. Here we provide a short version of the theorem for the sake of clarity; see Theorem \ref{T:Furstenberg discretizasion-cocompact Lattice} for the expanded version.

\begin{thm}\label{T:Furstenberg discretizasion-cocompact Lattice-The short version} 
Let $G$ be a locally compact, noncompact group, let $\Gamma$ be a cocompact lattice in $G$, and let $\eta\in \Bc(G)$. 
Then there exists a probability measure $\mu$ on $\Gamma$  with $\supp\, \mu=\Gamma$ such that every $(G,\eta)$-stationary space is also $(\Gamma,\mu)$-stationary.
Furthermore, $\mu$ can be constructed to have finite $m$-moment, for every $m\in \N$, with respect to any subadditive function $\sg$ on $G$ (after it is restricted to $\Gamma$). 
\end{thm}

We point out that a more general statement of the above theorem was claimed in \cite[Proposition 2.23]{badervigdorovich} without a detailed proof. However, this more general result is incorrect, as \cite[Section 1.4.(iv)]{benoistquint} gives a counterexample in the setting of non-cocompact lattices in $\mathrm{SL}_2(\R)$. We also do not see how the adaptation of the proof of \cite[Proposition VI.4.1]{Marg1} as suggested in order to obtain \cite[Proposition 2.23]{badervigdorovich} would work for our special case of cocompact lattices.

We also want to compare our result with the recent discretization theorem of Derimay \cite[Proposition~3.12]{derimay}. In the setting of reversible random walks on locally finite graphs, Derimay constructs discretizing measures with finite first moment and, for cocompact lattice actions, finite exponential moment with respect to the orbit metric. Our result is of a different nature: For cocompact lattices in locally compact groups, it yields Furstenberg discretizations with finite moments of every order with respect to any prescribed locally bounded subadditive function on the ambient group. Thus, Derimay's result is stronger for the natural orbit metric, whereas our construction is designed to allow more general ambient groups and more flexible moment controls. 
Finally, shortly before the submission of this article, Derimay informed the third author that he has extended the results of \cite{derimay} to the setting of general locally compact groups and expects to submit this work for publication in the near future. Since the methods employed in his approach appear to be substantially different from those used in the present paper, and in many respects complementary to them, we believe that there is significant value in having both proofs available in the literature.

\section*{Acknowledgements}
We are indebted to Amos Nevo for valuable discussions on Furstenberg discretizations in an early stage of this project and for bringing the question on moment conditions on the discretizations to our attention. We thank Uri Bader and Itamar Vigdorovich for discussions about the Furstenberg discretization and in particular about the issue regarding \cite[Proposition 2.23]{badervigdorovich}. The third author would also like to thank Antoine Derimay for informing him of this recent development and for explaining the ideas underlying his extension of \cite[Proposition~3.12]{derimay} to general locally compact groups.

\section{Preliminaries} \label{sec:preliminaries}

\subsection{Notation and conventions}\label{Notation and conventions}

Throughout this article, unless explicitly stated otherwise, $G$ always denotes a locally compact second countable (lcsc) group, $e$ its identity element, and $m$ a Haar measure on $G$. For convenience, integration of a function $f$ with respect to $m$ is denoted $\int_G f \, ds$.

A Borel probability measure $\eta$ on $G$ is called \textbf{non-degenerate} if its support generates a dense subsemigroup of $G$, and it is called \textbf{spread out} if it is non-degenerate and absolutely continuous with respect to $m$. The latter implies that $\eta$ has a well-defined density (or Radon-Nikodym derivative) $\rho:=\frac{d\eta}{ds}$.

The set $\B(G)$ denotes the spread-out Borel probability measures on $G$ with bounded density with respect to $m$, and $\Bc(G)$ denotes the spread-out Borel probability measures on $G$ with bounded density and compact support.

\begin{defn} \label{defn:lengthfunction}
A \textbf{subadditive function} on a locally compact group $G$ is a locally bounded (i.e.~bounded on compact sets) measurable function $\fL \colon G \to [0,\infty)$ satisfying  $\fL(st) \leq \fL(s) + \fL(t)$ for all $s,t \in G$. A subadditive function $\fL$ is called a \textbf{length function} if it also satisfies $\fL(e)=0$ and $\fL(s^{-1})=\fL(s)$ for all $s\in G$.
\end{defn}


Let $\alpha > 0$. A Borel probability measure $\eta$ on $G$ has \textbf{finite $\alpha$-moment} with respect to a subadditive function $\fL \colon G \to [0,\infty)$ if $\int_{G} \fL(s)^\alpha \, d\eta(s)<\infty$. If $\alpha=1$, the measure $\eta$ said to have \textbf{finite first moment} with respect to $h$ (or \textbf{finite first} $\fL$-\textbf{moment}).

\subsection{Random walks and entropy} \label{subsec:entropyrw}

Let $G$ be a lcsc group, $\eta$ a Borel probability measure on $G$, and $(Y_n)_{n \geq 1}$ a sequence of $G$-valued independent identically distributed random variables with law $\eta$. Set $W_0 := e$ and $W_n := Y_1\ldots Y_n$. Then $W_n$ is distributed with law $\eta^{*n}$ (the $n$-fold convolution product of $\eta$ with itself), where $\eta_1 \ast \eta_2$ is the unique Borel probability measure determined by $\int_G f(s) \, d(\eta_1 * \eta_2)(s) = \int_G \int_G f(tt') \, d\eta_1(t) \, d\eta_2(t')$. The random process $(W_n)_{n \geq 0}$ is called the $\eta$-\textbf{random walk} on $G$ (with starting point $e$).


\begin{defn}\label{D:differential entropy}
Let $G$ be a lcsc group, and $\eta$ a spread-out Borel probability measure on $G$. The $n^{\mathrm{th}}$ \textbf{differential entropy} of $\eta$ (which only depends on the law $\eta^{*n}$) is defined as
\[
	H(G,\eta^{*n}) := -\int_{G} \log \phi_n(s) \, d\eta^{*n}(s),
\]
where $\phi_n$ denotes the Radon-Nikodym derivative
\[
    \phi_n:=\frac{d\eta^{*n}}{ds}.
\] 
\end{defn}

If $H(G,\eta^{*n})$ is finite for every $n \in \N$, then, by \cite[page 253]{Derriennic II} (see also \cite[Theorem 2.5]{Forghani-Tiozzo}), the limit $\lim_{n\rightarrow \infty} \frac{1}{n} H(G,\eta^{\ast n}) =: h(G,\eta)$ exists. This limit is (mainly in the setting of countable discrete groups) referred to as the \textbf{Avez entropy} of the $\eta$-random walk.



For background and details on random walks and their entropy, we refer to \cite{Forghani-Tiozzo,Lyons-Peres,Zheng}.

\subsection{Stationary dynamical systems and the Poisson boundary} \label{subsec:stationarydynamicalsystemspoissonboundary}

Let $G$ be a lcsc group, let $\eta$ be a spread-out Borel probability measure on $G$, and let $(X,\xi)$ be a standard probability space equipped with a jointly measurable action $G \curvearrowright (X,\xi)$, i.e.~the map $G \times X \to X$ defining the action is measurable with respect to the product $\sigma$-algebra on $G \times X$. The measure $\xi$ is called $\eta$-\textbf{stationary} (and $(X,\xi)$ is called a $(G,\eta)$-\textbf{stationary space} and $(G,\eta) \curvearrowright (X,\xi)$ a \textbf{stationary dynamical system}) if
\[
    \eta * \xi = \xi;
\]
that is, $\int_G s\xi \, d\eta(s) = \xi$, where $s\xi(A)=\xi(s^{-1}A)$ for all Borel sets $A \subset X$. Stationary dynamical systems are automatically non-singular (i.e.~quasi-invariant).

A pivotal example of a stationary space is the \textbf{Poisson boundary} of an $\eta$-random walk on $G$. It is a probability space that reflects the asymptotic behaviour of the random walk, going back to Furstenberg \cite{Furstenberg1963a}; we define the Poisson boundary by means of the characterization given in \cite[Theorem 2.11]{badershalom} (see also \cite[Theorem 2.2]{boutonnethoudayer}).

Let $G$ be a lcsc group and $\eta$ a spread-out Borel probability measure on $G$. Let $H^{\infty}(G,\eta)$ denote the space of bounded measurable $\eta$-harmonic functions on $G$, i.e.~functions $h \in L^\infty(G)$ satisfying $h(s) = \int_G h(st) \, d\eta(t)$ for all $s \in G$, and consider the map
\[
    T \colon L^{\infty}(X,\xi) \to H^{\infty}(G,\eta), \quad f \mapsto Tf,
\]
where
\[
    Tf(s)=\int_X f(s x) \, d\xi(x).
\]
The Poisson boundary $(\Pi_{\mu},\nu_{\infty})$ is the unique (up to $G$-equivariant measure space isomorphism) $(G,\eta)$-stationary space $(X,\xi)$ for which the map $T$ is a $G$-equivariant isometric isomorphism.

For overviews of the Poisson boundary, including different characterizations, we also refer to \cite{Furstenberg1973,Kaim-Vers 1,Furm1,Forghani-Tiozzo}.

\subsection{Furstenberg entropy} The \textbf{Furstenberg entropy} of a $(G,\eta)$-stationary space $(X,\xi)$, which goes back to \cite{FurstenbergNCRP} (see also \cite{NevoZimmer,Kaim-Vers 1}), is defined as
\begin{align*} 
  h_\eta(X,\xi) := - \int_{G} \int_X  \log \rho(s,x) \, d\xi(x) d\eta(s),
\end{align*}
where
\begin{align} \label{Eq:Radon-Nikodym cocycle-defn}
    \rho(s,x):=\frac{d(s^{-1}\xi)}{d\xi}(x), \;\; s\in G, \; x\in X.
\end{align}
It is well known that $\rho$ satisfies the following cocycle relation:
\begin{align} \label{Eq:Radon-Nikodym cocycle-cocycle relations}
    \rho(st,x)=\rho(s,tx)\rho(t,x) \;\; \textrm{for all} \;\; s,t\in G \textrm{ and } x\in X.
\end{align}

We have $h_\eta(X,\xi)\geq 0$, with equality if and only if $\xi$ is $G$-invariant. Also, 
\begin{align} \label{Eq:maximality of Furstenberg boundary}
h_\eta(X,\xi) \leq h_\eta(\Pi_{\eta},\nu_{\infty}),
\end{align}
where $(\Pi_{\eta},\nu_{\infty})$ is the Poisson boundary of the $\eta$-random walk and $(X,\xi)$ is an arbitrary $(G,\eta)$-stationary space. Note that \eqref{Eq:maximality of Furstenberg boundary} is often stated in the literature with $(X,\xi)$ being a $\eta$-\emph{boundary}. However, from \cite[Theorem 4.3]{Furs-Glas 1}, it follows that it also holds for arbitrary $(G,\eta)$-stationary spaces. Moreover, for every $\eta$-boundary $(X,\xi)$, equality in \eqref{Eq:maximality of Furstenberg boundary} occurs if and only if there is a measurable $G$-equivariant isomorphism from $(X,\xi)$ onto $(\Pi_{\eta},\nu_{\infty})$. 

We close this section by recalling the following relation (see \cite{Kaim-Vers 1}) between the Avez entropy of the $\eta$-random walk and the Furstenberg entropy of the associated Poisson boundary: 
\begin{align} \label{Eq:Furs entropy as limit of diff entopy}
    h_\eta(\Pi_{\eta},\nu_{\infty}) = \lim_{n\rightarrow \infty} \frac{1}{n} H(G,\eta^{\ast n}).
\end{align}

\subsection{Weights on groups} \label{subsec:weightsongroups}
\begin{defn} \label{D:weight}
A \textbf{weight} on a locally compact group $G$ is a locally bounded (Borel) measurable function $\omega \colon G \rightarrow [a,\infty)$, for some $a>0$, that is weakly submultiplicative, i.e.~there exists $C>0$ such that for all $s,t\in G$,
\begin{align*} 
\om(st)\leq C \om(s)\om(t).
\end{align*}
Two weights $\om$ and $\widetilde{\om}$ on $G$ are said to be equivalent if there are $C_1,C_2>0$ such that $C_1\,\om \leq \widetilde{\om} \leq C_2\,\om$.
\end{defn}

\begin{rem} \label{R:growth of weights}
\
\begin{enumerate}[(i)]
\item Every weight $\om$ on $G$ is equivalent to a submultiplicative weight. Also, we can always rescale a weight to an equivalent weight with values in $[1,\infty)$. 
\item Weights on (compactly generated) groups are closely related to the notion of \emph{gauge}, as considered in \cite{Derriennic II} (see also \cite[Definition 5.1]{Forghani-Tiozzo}). 
Indeed, given a temperate subadditive gauge $\mathcal{G}=(\mathcal{G}_n)_{n \in \N}$ on $G$, the function $\om(s):=e^{|s|_\mathcal{G}}$, where $|s|_\mathcal{G}=\min\{n : s \in \mathcal{G}_n\}$, defines a submultiplicative weight on $G$. On the other hand, if $G$ is compactly generated with compact symmetric generating neighbourhood $Q$ of $e$ and $\om$ is a submultiplicative weight on $G$, then $\mathcal{G}_n:=\{s\in G : \om(s)\leq C^n \}$ defines a temperate subadditive gauge on $G$, where $C=\sup\{\om(s) : s\in S\}$. 
\end{enumerate}
\end{rem}

\begin{defn} \label{exm:polynomial weights}
Let $\fL \colon G\to [0,\infty)$ be a length function on $G$. For $d \geq 1$, the function
\[
    \mathcal{P}_{\fL}^{d} \colon s \mapsto (1+\fL(s))^d
\]
defines a weight on $G$. It is called the \textbf{polynomial weight} of degree $d$ with respect to $\fL$.
\end{defn}

\subsection{Weighted $L^p$-spaces}
Let $(X,\nu)$ be a $\sg$-finite measure space, and $\om \colon X \to (0,\infty)$ a positive measurable function on $X$. Such a function is called a weight function on the space $X$. The \textbf{weighted $L^p$-space} $L^p(X,\om)$ is defined as
\[
L^p(X,\omega) := L^p(X,\om^p \nu) = \{f \colon X \to \C \textrm{ measurable} : f\omega \in L^p(X,\nu)\}
\]
equipped with the norm
\[
    \|f\|_{p,\om} = \left( \int_X |f|^p \om^p d\nu \right)^{\frac{1}{p}}.
\]
Weighted $L^p$-spaces are very standard in functional analysis. In this article, we mainly consider weighted $L^p$-spaces for which the underlying measure space is a group; weights on groups (as covered in Section \ref{subsec:weightsongroups}) are indeed a special class of weight functions on measure spaces. 

We recall a result on weighted $L^p$-spaces \cite[Proposition 2.1]{ASdLSW1}; its proof uses (complex) interpolation methods.


\begin{prop} \label{prop:interpolationconsequence}
Let $(X,\nu)$ be a $\sg$-finite measure space, $\omega \colon X \to (0,\infty)$ a function, and $1< q_0\leq p_0< \infty$ with $\dfrac{1}{p_0}+\dfrac{1}{q_0}=1$. For every
$f\in L^1(X,\nu)\cap L^{q_0}(X,\omega^{1/p_0})$
with $\Vert f\Vert_{1}=1$, the mapping
\begin{equation*}
[p_0,\infty) \rightarrow \R, \ \ p \mapsto -p \, \log \Vert f\Vert_{q,\omega^{1/p}},
\end{equation*}
with $1< q \leq q_0 \leq p_0 \leq p < \infty$ with $\dfrac{1}{p}+\dfrac{1}{q}=1$, is an increasing function. Equivalently, we have that the mapping
\begin{equation*}
(1,q_0] \rightarrow \R, \ \ q \mapsto -p \, \log \Vert f\Vert_{q,\omega^{1/p}}
\end{equation*}
is a decreasing function.
\end{prop}

In the special setting of the underlying measure space being a group $G$ with weight $\om$, the space $L^1(G,\om):=\{f : f\om\in L^1(G)\}$, with the norm $\|f\|_{1,\om}:=\int_{G} |f(s)| \, \om(s) \, ds$, is a Banach algebra with respect to pointwise addition and convolution. This algebra is called the \textbf{weighted $L^1$-algebra} on $G$ with weight $\om$.



\subsection{The Kunze--Stein property}

\begin{defn}
    A locally compact group $G$ has the \textbf{Kunze--Stein property} if the convolution product $m \colon C_c(G) \times C_c(G) \to C_c(G), \; (f,g) \mapsto f \ast g$ extends to a bounded bilinear map $L^q(G) \times L^2(G) \to L^2(G)$ for all $q \in [1,2)$.
\end{defn}
Kunze and Stein established this property for the group $\mathrm{SL}(2,\mathbb{R})$ \cite{MR0163988}. Cowling generalized this result and proved that connected semisimple Lie groups with finite center have it \cite{MR0507240}. Additional classes of Kunze-Stein groups were provided in \cite{MR0507240}, \cite{MR0936361} and \cite{veca}.

\subsection{The property of rapid decay}

\begin{defn} \label{D:RD}
    A locally compact group $G$ has the \textbf{property of rapid decay} (also called \textbf{property RD}) with respect to a length function $\mathcal{L}$ if there exists $d \in \N$ such that the canonical inclusion $C_c(G) \hookrightarrow C^*_{r}(G)$ extends to a bounded linear map $L^2(G,\om) \to C^*_{r}(G)$, where $C^*_{r}(G)$ is the reduced $C^*$-algebra of $G$ and $\om$ is the weight on $G$ given by $\om(s)=(1+\mathcal{L}(s))^d$.
\end{defn}


The property of rapid decay was first established for the non-abelian free groups $\mathbb{F}_n$, with $n \geq 2$, by Haagerup \cite{Haagerup}; its systematic study goes back to \cite{Joli}.

In the continuous setting, it is known that semisimple Lie groups with finite center have property RD \cite[Th\'eor\`eme 1]{herz1}. More generally, a connected Lie group has property RD if and only if its solvable radical has polynomial volume growth \cite[Theorem 0.1]{chatterjipittetsaloffcoste}.

\section{Computing the Lyapunov exponent and entropy}

We recall the Lyapunov exponent of a random walk on a locally compact group and prove a vanishing result for this quantity. Subsequently, we study the notion of weighted differential entropy, and we highlight its relation with the Lyapunov exponent. Finally, we generalize the notion of convolution entropy to the locally compact setting. These results are mostly generalizations of the discrete case, as covered recently in \cite{ASdLSW1}. Most of the proofs are omitted, as they are straightforward modifications of the discrete case.

\subsection{Lyapunov exponent and its vanishing}
As before, $G$ denotes a locally compact second countable group. 

\begin{defn}\label{D:Lyap expo}
Let $\om$ be a weight on $G$, and let $\eta$ be a Borel probability measure with finite first $\log \om$-moment (that is, $\displaystyle\int_{G} \log \om (s)d\eta(s) < \infty$). The \textbf{Lyapunov exponent} of $\eta$ with respect to $\om$ is defined as
\begin{align*}
\Lc_\om(G,\eta):=\displaystyle\lim_{n\rightarrow \infty}\dfrac{1}{n}\int_{G}\log \omega (s) \, d\eta^{\ast n}(s).
\end{align*}
\end{defn}
Since this is an asymptotic quantity, it may be more appropriate to view this quantity as an invariant of the $\eta$-random walk rather than of the measure $\eta$.

To the knowledge of the authors, this version of the Lyapunov exponent for random walks on groups goes back to the work of Nevo \cite{Nevo1}. Nevo attributes this notion of Lyapunov exponent to Margulis \cite[pp.~189-190]{Marg1}, who uses it implicitly and who in turn refers to \cite{G1,Ledrappier1982} for other uses of this quantity.

Let $\om$ be a weight on $G$ and $\eta$ a Borel probability measure on $G$ with finite first $\log \om$-moment. Analogous to the discrete case, it can be shown that the Lyapunov exponent exists and is well defined (see \cite[Lemma 3.2]{ASdLSW1}). Also, $\Lc_{\om}(G,\eta)$ is a non-negative number, and if $\widetilde{\om}$ is a weight that is equivalent to $\om$, then $\Lc_{\widetilde{\om}}(G,\eta)=\Lc_\om(G,\eta)$ (see \cite[Proposition 3.3]{ASdLSW1}). 

Finite first $\log \om$-moment ensures the existence of the Lyapunov exponent of $\eta$. If we consider a slightly stronger integrability condition, we obtain an alternative formula for $\Lc_\om(G,\eta)$ in terms of the asymptotic behaviour of the spectral radius of $\eta$ in an appropriate ascending class of weighted group algebras. The proof of this result in the discrete setting is given in \cite[Theorem 3.5]{ASdLSW1}. First, given a weight $\om$ on $G$ and $p\geq 1$, let $\om_p$ denote the weight defined by 
\begin{align}\label{Eq:defn-weight interpolation in Lp}
    \om_p(s) := \om(s)^{\frac{1}{p}}.
\end{align}
\begin{prop} \label{prop:Lyap expo-weighted L1}
Let $\om$ be a weight on $G$, and let $\eta$ be a Borel probability measure on $G$ with $\eta \in \displaystyle\bigcup_{p\geq 1}L^1(G,\omega_p)$. Then $\eta$ has a finite Lyapunov exponent with respect to $\om$, and we have
\begin{align*} 
\Lc_\om(G,\eta):=\lim_{p \rightarrow \infty} p\log r_{L^1(G,\omega_p)}(\eta),
\end{align*}
where $r_{L^1(G,\omega_p)}(\eta)$ is the spectral radius of $\eta$ in the Banach algebra $L^1(G,\omega_p)$.
\end{prop}

We finish this section by highlighting a vanishing result for the Lyapunov exponent. 
The proof is analogous to the case of discrete group given in \cite[Theorem 3.6]{ASdLSW1}, so again, we omit it. We wish to highlight, however, that it relies on a major classical result in harmonic analysis and potential theory, namely the theorem of de la Vall\'{e}e Poussin. For more details, we refer to \cite[Theorems 3.6 and 3.7 and Appendix A]{ASdLSW1}.

\begin{thm} \label{T:vanishing Lyap expo-log of length function}
Let $\fL$ be a subadditive function on $G$, and let $\eta$ be a probability measure on $G$ with finite $\log(1+\fL)$-moment. Then the Lyapunov exponent of $\eta$ with respect to $1+\fL$ vanishes. In other words,
\begin{align*}
\Lc_{1+\fL}(G,\eta)=\displaystyle\lim_{n\rightarrow \infty}\dfrac{1}{n}\int_{G} \, \log (1+\fL(s))\,d\eta^{\ast n}(s)=0.
\end{align*}
\end{thm}


\subsection{Weighted entropy}

We consider a generalization of the concept of differential entropy by taking into account a weight on the underlying group; this is analogous to the concept of a weighted Shannon entropy, as considered in \cite[Definition 4.1]{ASdLSW1}.  

\begin{defn}\label{D:weighted Shannon entropy}
Let $\om$ be a weight on $G$, and let $\eta$ be a probability measure on $G$ with $\eta\in \BG$. 
The \textbf{weighted differential entropy} of the $\eta$-random walk at step $n$ with respect to $\om$ is defined as
\begin{equation}\label{Eq:weighted Shannon entropy-defn}
H_{\omega}(G,\eta^{*n}):=-\int_G \log [\phi_n(s)\omega(s)] \, d\eta^{*n}(s),
\end{equation}
where, as before, $\phi_n$ is the Radon--Nikodym derivative $\phi_n := \frac{d\eta^{*n}}{ds}$. 
\end{defn}


\begin{rem} \label{R:defn weighted entropy}
The weighted differential entropy of a probability measure need not exist. However, we can (formally) rewrite \eqref{Eq:weighted Shannon entropy-defn} as
\begin{equation}\label{Eq:weighted entropy vs entropy and log moment}
H_{\omega}(G,\eta^{*n})=H(G,\eta^{*n})-\displaystyle\int_G\log \omega (s)\,d\eta^{*n}(s),
\end{equation}
so the $n^{\mathrm{th}}$ weighted differential entropy exists if $\eta$ has both finite $n^{\mathrm{th}}$ differential entropy and finite $\log \om$-moment. 
\end{rem}

The following result presents the generalization to the locally compact setting of \cite[Theorem 4.4]{ASdLSW1}. Recall that given a weight $\om$ on $G$ and $p \geq 1$, the weight $\om_p$ is given by $\om_p(s):=\om(s)^{\frac{1}{p}}$.

\begin{thm} \label{T:compare weighted and non-weighted Shannon entropy}
Let $\om$ be a weight on $G$, and let $\eta$ be a probability measure on $G$ such that $\eta\in \B(G)\cap L^{q_0}(G,\omega_{p_0})$, with $1< q_0\leq p_0 <\infty$ and $\frac{1}{p_0}+\frac{1}{q_0}=1$.
\begin{enumerate}[(i)]
    \item The mapping
\[
[p_0,\infty) \rightarrow \R,\ \ p \mapsto -p \log \Vert \eta\Vert_{q,\omega_{p}},
\]
where $p$ and $q$ satisfy $\dfrac{1}{p}+\dfrac{1}{q}=1$, is an increasing function on $[p_0,\infty)$.
\item If, in addition, $\eta$ has both finite differential entropy and finite $\log \om$-moment, then
\begin{equation} \nonumber
H_{\omega}(G,\eta)=\displaystyle\lim_{p\rightarrow \infty} -p \log \Vert \eta \Vert_{q,\omega_{p}},
\end{equation}
where $\dfrac{1}{p}+\dfrac{1}{q}=1$.
\end{enumerate}
\end{thm}


\begin{defn} \label{D:weighted Avez entropy}
Let $\om$ be a weight on $G$, and let $\eta$ be a probability measure on $G$ with $\eta\in \BG$ and both finite $n^{\mathrm{th}}$ differential entropy for all $n\in \N$ and finite $\log \om$-moment. The \textbf{weighted asymptotic entropy} of $\eta$ with respect to $\omega$ is defined as
\begin{equation} \label{Eq:weighted Avez entropy-defn}
h_{\omega }(G,\eta):= \displaystyle\lim_{n\rightarrow \infty} \dfrac{H_{\omega }(G,\eta^{\ast n})}{n}.
\end{equation}
\end{defn}

\begin{rem}
Let $\om$ be a weight on $G$, and let $\eta$ be as in Definition \ref{D:weighted Avez entropy}. By Remark \ref{R:defn weighted entropy}, the mapping
$$\N\to \R, \ \ n\mapsto H_{\omega }(G,\eta^{\ast n})$$
is well defined. Moreover, since
$$H_{\omega}(G,\eta^{\ast n})=H(G,\eta^{\ast n})-\displaystyle\int_G\log \omega (s)\,d\eta^{{\ast n}}(s),$$
Definition \ref{D:Lyap expo} and \eqref{Eq:Furs entropy as limit of diff entopy} imply that the limit in \eqref{Eq:weighted Avez entropy-defn} exists, and we have 
\begin{align} \label{Eq:weighted Avez entropy-difference entropy and Lyap expo}
\begin{split}
 h_{\om}(G,\eta) &=
  \lim_{n\rightarrow \infty} \frac{1}{n}H(G,\eta^{\ast n})- \lim_{n\rightarrow \infty} \frac{1}{n} \displaystyle\int_G\log \omega (s)\,d\eta^{{\ast n}}(s) \\
  &= h_\eta(\Pi_{\eta},\nu_{\infty})-\Lc_\om(G,\eta),
\end{split}
\end{align}
where, as before, $(\Pi_{\eta},\nu_{\infty})$ is the Poisson boundary of $(G,\eta)$ and $h_\eta(\Pi_{\eta},\nu_{\infty})$ is its Furstenberg entropy. In other words, the weighted asymptotic entropy of the $\eta$-random walk is the difference of Furstenberg entropy of the Poisson boundary of $(G,\eta)$ and the Lyapunov exponent with respect to $\om$. In particular:
\begin{align} \label{Eq:equality of weighted and nonweighted Avez entropy}
h_{\omega }(G,\eta)=h_\eta(\Pi_{\eta},\nu_{\infty}) \; \Longleftrightarrow \; \Lc_\om(G,\eta)=0.
\end{align}
\end{rem}

There is another consequence of the relations \eqref{Eq:weighted entropy vs entropy and log moment} and \eqref{Eq:weighted Avez entropy-difference entropy and Lyap expo} that is crucial for our approach. If the weighted differential entropy is a non-positive number, then the differential entropy becomes dominated by the $\eta$-integral of $\log \om$. This simple fact is useful and we take advantage of it. For instance, it occurs whenever the inverse of the weight is $L^1$-integrable. 
The argument is straightforward and similar to \cite[Theorem 4.7]{ASdLSW1}; cf.~\cite[Lemma 5.3 and Corollary 5.4]{Forghani-Tiozzo}. 

\begin{thm} \label{T:Lyap Expo dominates entropy-inverse integrable weights}
Let $\eta$ be a probability measure on $G$ with $\eta\in \BG$, and let $\om$ be a positive measurable function on $G$ such that $\omega^{-1} = 1/\om \in L^1(G)$ and 
\begin{align*}
\int_{G}\log \omega (s) \, d\eta^{*n}(s) < \infty \textrm{ for all } n \in \N.
\end{align*}
Then the $\eta$-random walk has finite $n^{\mathrm{th}}$ differential entropy for all $n\in \N$, and
\begin{align*}
h_\eta(\Pi_{\eta},\nu_{\infty})\leq \liminf_{n\to \infty} \frac{1}{n}\int_{G}\log \omega (s) \, d\eta^{\ast n}(s).
\end{align*}

\end{thm}


\subsection{Convolution entropy and Furstenberg entropy} \label{subsec:convolutionentropy}


For $1 < p < \infty$, let $\lambda_p \colon L^1(G)\to B(L^p(G))$ be the left-regular representation, i.e.~the representation defined by
$$ \lambda_p(f)g=f*g.$$
The representation $\lambda_p$ is a contractive representation and the norm closure of $\lambda_p(L^1(G))$ inside $B(L^p(G))$ is called the algebra $\pf_p(G)$ of \textbf{$p$-pseudofunctions}, which goes back to \cite{Herz}. For simplicity, we write $\lambda=\lambda_2$, which is the left-regular representation of $L^1(G)$ on $L^2(G)$. As described in \cite[Section 2.3]{ASdLSW1}, we have
\begin{align}\label{Eq:complex interpolation-convlution operators-norm relations}
  \|f\|_{\pf_u(G)}\leq \|f\|_1^{1-\theta}\|f\|_{\pf_p(G)}^\theta, \,\, \text{with} \,\, \theta=\frac{u}{p},
\end{align}
for $1\leq u\leq p$ and $f\in L^1(G)$.

The following result states that a certain asymptotic behaviour of a probability measure $\eta$ in $\pf_p(G)$ is dominated by the Furstenberg entropy of the $\eta$-random walk on its Poisson boundary. This is a generalization of a similar result, shown in \cite[Theorem 5.1]{ASdLSW1}, for discrete groups. We present the proof here, as we have to make a modest adjustment in the locally compact setting.    

\begin{thm}\label{T:Avez entropy-p conv operators-I}
Let $\eta$ be a probability measure on $G$ with $\eta\in \BG$ such that $\eta$ has finite $n^{th}$ differential entropy, for all $n\in \N$. Then the mapping
\begin{align}\label{Eq:asymptote Lp conv operators}
[2,\infty )\rightarrow [0,\infty ), \;\; p \mapsto -p\log r_{\pf_q(G)}(\eta),
\end{align}
where $\frac{1}{p}+\frac{1}{q}=1$, is increasing and bounded on $[2,\infty)$, and
\begin{align*} 
 \displaystyle\lim_{p\rightarrow \infty}-p\log r_{\pf_q(G)}(\eta) \leq h_\eta(\Pi_{\eta},\nu_{\infty}).
\end{align*}
\end{thm}

We call the limit $\lim_{p\rightarrow \infty}-p\log r_{\pf_q(G)}(\eta)$ the \textbf{convolution entropy} of the $\eta$-random walk, denoted by $c(G,\eta)$ 
(see also \cite[Definition 1.3]{ASdLSW1}).

\begin{proof}
We first note that since $\eta\in L^1(G)\cap L^\infty(G)$, we have $\eta\in L^q(G)$ for $q\in [1,\infty]$. For every $1<q<p<\infty $ with $\dfrac{1}{p}+\dfrac{1}{q}=1$ and $n\in \mathbb{N}$ with $n\geq 2$, we have
$$\Vert \eta^{\ast n}\Vert_{q}^{-p} = \Vert \eta^{\ast (n-1)} \ast \eta\Vert_{q}^{-p} \geq \Vert \eta \Vert_q^{-p} \Vert \eta^{\ast (n-1)}\Vert_{\pf_q(G)}^{-p}  \geq \Vert \eta \Vert_q^{-p} \Vert \eta \Vert_{\pf_q(G)}^{-(n-1)p} .$$
Hence,
$$-p\log \Vert \eta^{\ast n}\Vert_{q} +p\log \Vert \eta\Vert_{q}\geq -(n-1)p \log \Vert \eta \Vert_{\pf_q(G)}\geq 0. $$
By letting $q\rightarrow 1^{+}$ (and hence, $p\rightarrow \infty$), we obtain
$$H(G,\eta^{\ast n})-H(G,\eta) \geq n \ A(G,\eta),$$
where
$$A(G,\eta) := \displaystyle\lim_{p\rightarrow \infty}-p\log \Vert \eta \Vert_{\pf_q(G)}.$$
We note that $A(G,\eta)$ exists, since by \eqref{Eq:complex interpolation-convlution operators-norm relations} the mapping
$$[2,\infty ) \mapsto [0,\infty ), \;\; p\mapsto -p \log \Vert \eta \Vert_{\pf_q(G)}$$
is increasing and bounded by $h_\eta(\Pi_{\eta},\nu_{\infty})$. Therefore,
$$h_\eta(\Pi_{\eta},\nu_{\infty})=\displaystyle\lim_{n\rightarrow \infty }\dfrac{H(G,\eta^{\ast n})}{n} \geq A(G,\eta).$$
By replacing $\eta$ with $\eta^{\ast n}$, we obtain
\begin{align*}
h_\eta(\Pi_{\eta},\nu_{\infty}) =\dfrac{h_\eta(\Pi_{\eta^{\ast n}},\nu_{\infty})}{n} \geq \dfrac{A(G,\eta^{\ast n})}{n} \geq \dfrac{-p\log \Vert \eta^{\ast n}\Vert_{\pf_q(G)}}{n}.
\end{align*}
Combining the preceding relation with \eqref{Eq:complex interpolation-convlution operators-norm relations}, it follows that the mapping \eqref{Eq:asymptote Lp conv operators} is increasing, with
\begin{align*}
& h_\eta(\Pi_{\eta},\nu_{\infty}) \geq -p\log r_{\pf_q(G)}(\eta ) \;\; \textrm{for} \;\; p\geq 2.
\end{align*}
Therefore,
\begin{align} \label{eq:avezentropyequalconvolutionentropy}
h_\eta(\Pi_{\eta},\nu_{\infty}) \geq \displaystyle\lim_{p\rightarrow \infty}-p\log r_{\pf_q(G)}(\eta ).
\end{align}
\end{proof}

The following theorem proves that for random walks on groups with property RD or Kuntz-Stein property, we have equality in \eqref{eq:avezentropyequalconvolutionentropy}.

\begin{thm}\label{T:Avez entropy-p conv operators-RD}
Let $G$ be a second countable locally compact group, and let $\eta$ be a probability measure on $G$ with $\eta\in \BG$ such that $\eta$ has finite $n^{th}$-differential entropy, for all $n\in \N$.  Suppose that one of the following holds:
\begin{enumerate}[(i)]
    \item $(G,\fL)$ has property RD and $\eta$ has a finite $\alpha$-moment with respect to $1+\fL$ for some $\alpha>0$.
    \item $G$ has Kunze-Stein Property.
\end{enumerate}
Then the Furstenberg entropy of the $\eta$-random walk on its Poisson boundary and the convolution entropy of the $\eta$-random walk coincide:
\begin{align} \label{Eq:Avez-entropy-p conv operators-RD}
   h_\eta(\Pi_{\eta},\nu_{\infty})=c(G,\eta).
\end{align}
\end{thm}

\begin{proof}
By Theorem \ref{T:Avez entropy-p conv operators-I}, we only need to prove that $ h_\eta(\Pi_{\eta},\nu_{\infty}) \leq c(G,\eta)$.
Suppose that $G$ has property RD with respect to a length function $\fL$. Then there exists $d \in \N$ such that for the weight $\om$ on $G$ given by
$$\om(s)=(1+\fL(s))^d, \;\; s\in G,$$
we have the inclusion $L^2(G,\om_2)\subseteq C^*_r(G)$.
By hypothesis, $\eta\in L^1(G,\om^{\frac{\alpha}{d}})$, so that by Proposition \ref{prop:Lyap expo-weighted L1}, $\eta$ has finite $\log \om$-moment. On the other hand, $\log \om=d\log(1+\fL)$, and hence, by Theorem \ref{T:vanishing Lyap expo-log of length function}, the Lyapunov exponent of $\eta$ with respect to $\om$ vanishes:
$$\Lc_\om(G,\eta)=0.$$
In other words, by \eqref{Eq:equality of weighted and nonweighted Avez entropy}, we have
$$ h_\eta(\Pi_{\eta},\nu_{\infty})=h_\omega(G,\eta).$$
Therefore, it suffices to prove that
\begin{align*}
 h_\omega(G,\eta)\leq \displaystyle\lim_{p\rightarrow \infty}-p\log r_{\pf_q(G)}(\eta).
\end{align*}
Now pick $p_0\in [2,\infty]$ such that $p_0> \displaystyle\frac{d}{\alpha}$. By our hypothesis,
$\eta \in L^{1}(G,\omega_{p_0})\cap L^\infty(G)$.
Equivalently, setting $\nu$ to be the Borel measure on $G$ given by $d\nu=\omega_{p_0}d\lambda$ where $\lambda$ denotes the Haar measure of $G$, we have that $\eta\in L^1(G,\nu)\cap L^\infty(G,\nu)$. Consequently, for $q > 1$,
\begin{align}\label{Eq:111}
    \eta \in L^q(G,\nu)=L^{q}(G,\om_{qp_0}).
\end{align}
\textcolor{blue}{}
Fix $1< q\leq p<\infty$ with $p\geq 2p_0$ and $\dfrac{1}{p}+\dfrac{1}{q}=1$.
Since $p\geq 2p_0$, we must have $\om_p\leq \om_{qp_0}$. Hence, by \eqref{Eq:111}, 
\begin{align*}
    \eta \in \bigcap_{p\geq 2p_0} L^{q}(G,\om_{p}).
\end{align*}
The rest of the proof follows similar to the proof of \cite[Theorem 5.2]{ASdLSW1}.

Now suppose that $G$ has Kunze-Stein property. Fix $q_0\in (1,2)$, and let $1< q<q_0<p<\infty$ with $\dfrac{1}{p}+\dfrac{1}{q}=1$. Using interpolation methods, it is shown in the proof of \cite[Proposition 5.1]{SW-exotic} that for $\theta,q'\in (0,1)$ satisfying 
\begin{align}\label{Eq:112}
  \theta=\frac{2}{p} \ \ \text{and} \ \ \frac{1}{q'}=1-\theta+\frac{\theta}{q_0},
\end{align}
we have
$$L^{q'}(G) \subseteq \pf_q(G).$$
Moreover, there exists a constant $C>0$ (in fact, $C$ is the norm of the inclusion map $L^{q_0}(G) \hookrightarrow C^*_r(G)$) such that for every $n\in \mathbb{N}$,
$$r_{\pf_q(G)}(\eta^{*n})\leq \Vert \eta^{\ast n}\Vert_{\pf_q(G)}\leq C^{\frac{2}{p}}\Vert \eta^{\ast n}\Vert_{L^{q'}(G)},$$
so that
$$-2\log C -p\log \Vert \eta^{\ast n}\Vert_{L^{q'}(G)} \leq -p n \log r_{\pf_q(G)}(\eta). $$
On the other hand, it follows from \eqref{Eq:112} that 
$$\frac{p}{p'}=2\left(1-\frac{1}{q_0}\right), \ \text{where}\ \ p'=\frac{q}{q-1}.$$
By combining the preceding relations, we obtain
$$-2\log C -2\left(1-\frac{1}{q_0}\right)p'\log \Vert \eta^{\ast n}\Vert_{L^{q'}(G)} \leq -p n \log r_{\pf_q(G)}(\eta). $$
Now, by letting $p\to \infty$, we have (by Theorem \ref{T:compare weighted and non-weighted Shannon entropy})
$$-2\log C + 2\left(1-\frac{1}{q_0}\right)H(\eta^{\ast n},G) \leq n\lim_{p\rightarrow \infty} -p \log r_{\pf_q(G)}(\eta). $$
Finally, by dividing both sides by $n$ and letting $n\to \infty$ and $q_0\to 2^-$, we obtain
\begin{align*}
 h_\eta(\Pi_{\eta},\nu_{\infty})=\displaystyle\lim_{n\rightarrow \infty}\dfrac{H(G,\eta^{\ast n})}{n}\leq \displaystyle\lim_{p\rightarrow \infty}-p\log r_{\pf_q(G)}(\eta).
\end{align*}
This completes the proof.
\end{proof}

\begin{rem}
\
\begin{enumerate}[(i)]
\item To the knowledge of the authors, equality \eqref{Eq:Avez-entropy-p conv operators-RD} is new, even for semisimple Lie groups. It generalizes known results on the entropy of random walks on groups with polynomial growth to non-amenable groups with property RD or Kunze-Stein. It also provides a new interesting way to compute the entropy of random walks on groups with property RD or Kunze-Stein. 

\item As a consequence of Theorem \ref{T:Avez entropy-p conv operators-RD}, we observe that for suitable probability measures on groups with property RD or Kunze-Stein, we can ``commute the double limits'' in computing the Furstenberg entropy of the $\eta$-random walk on its Poisson boundary. Indeed, for a probability measure $\eta$ satisfying \eqref{Eq:Avez-entropy-p conv operators-RD}, we have
\begin{align*}
 h_\eta(\Pi_{\eta},\nu_{\infty}) &= \lim_{n\rightarrow \infty} \lim_{q\rightarrow 1^{+}} -\frac{p}{n}\ln \Vert \eta^{*n} \Vert_{q} \\
 &= \lim_{q\rightarrow 1^+} \lim_{n\rightarrow \infty} -\frac{p}{n}\ln \Vert \eta^{*n} \Vert_{q}.
\end{align*}
\end{enumerate}
\end{rem}

\section{Poisson boundaries and amenable stationary actions} \label{sec:poissonboundariesstationarydynamicalsystems}

\subsection{$L^p$-Koopman representation}
Let $(X,\xi)$ be a standard probability space equipped with a non-singular (also called quasi-invariant) measurable $G$-action. 
For $1 \leq p < \infty$, this action induces the {\textbf{$L^p$-Koopman representation}} $$\pi_{p,X}: G\to B(L^p(X,\xi))$$
defined by
\begin{align*} 
[\pi_{p,X}(s)f](x)=\left[\frac{d(s\xi)}{d\xi}(x)\right]^{1/p} f(s^{-1}x), \qquad s \in G, \; f\in L^p(X,\xi) \textrm{ and }\xi\mathrm{-a.e.}~x,
\end{align*}
where $\frac{d(s\xi)}{d\xi}$ denotes the Radon-Nikodym derivative. In the unitary case, we write $\pi_X:=\pi_{2,X}$.


Following \cite{Nevo1}, we consider the function $\Xi_{X} \colon G \to [0,\infty)$, coming from the action $G \curvearrowright (X,\xi)$, defined as the diagonal matrix coefficient of $\pi_{X}$ with respect to the unit vector $1\in L^\infty(X,\xi) \subseteq L^2(X,\xi)$, i.e.
\begin{align} \label{Eq:Harich-Chandra function}
  \Xi_{X}(s):=\la \pi_{X}(s) 1, 1 \ra=\int_{X} \rho(s,x)^{1/2}  d\xi(x), \;\; s\in G,
\end{align}
where $\rho(s,x)$ is as in \eqref{Eq:Radon-Nikodym cocycle-defn}.
This function is a generalization of the \textbf{Harish-Chandra $\Xi$-function} to the dynamical setting. It is straightforward to see that $\Xi_X \geq 0$ and $\Xi_{X}(s^{-1})=\Xi_{X}(s)$ for all $s\in G$.



\subsection{Furstenberg entropy and decay of the Harish-Chandra $\Xi$-function}

Let $(X,\xi)$ and $(Y,\eta)$ be measurable $G$-spaces. A measurable map $\phi \colon (X,\xi)\to (Y,\eta)$ is called a \textbf{factor} (and $(X,\xi)$ is called an \textbf{extension} of $(Y,\eta)$) if $\phi$ is $G$-equivariant and $\eta$ is the push-forward measure of $\xi$ under $\phi$. 
If $(X,\xi)$ and $(Y,\eta)$ are measurable $G$-spaces and $\phi \colon (X,\xi)\to (Y,\eta)$ is a factor, then the map $\phi$ is \textbf{measure-preserving} if for every $s \in G$, we have $s\xi_y=\xi_{sy}$ for $\eta$-a.e.~$y$. Here, the probability measure $\xi_y$ is defined through the disintegration $\xi=\displaystyle\int_{y\in Y} \xi_yd\eta(y)$.

It is well known that the factor $\phi \colon (X,\xi)\to (Y,\eta)$ is measure preserving if and only if $h_\eta(X,\xi)=h_\mu(Y,\eta)$ (see \cite[Theorem 1.3.(2)]{Furs-Glas 1}).

We are now ready to present one of the main results of this article, which shows that the sufficiently fast decay of $\Xi_X$ (with respect to a length function on $G$) ensures that the Furstenberg entropy of $(X,\xi)$ must be equal to the Furstenberg entropy of its Poisson boundary. In fact, we present our result in a more general setting that will be useful for us later. This result is also a generalization of \cite[Theorem 6.2]{ASdLSW1}, which stated a weaker version for countable discrete groups with rapid decay.  


\begin{thm} \label{T:Square integrable-inverse Harich Chandra function-the general case}
Let $G$ be a second countable locally compact group with a subadditive function $\fL$, and let $\eta$ be a probability measure on $G$ with $\eta \in \BG$ such that $\eta$ has both finite $n^{th}$-differential entropy for all $n\in \N$ and finite $\log(1+\fL)$ moment. Let $(X,\xi)$ be a standard probability space equipped with a non-singular measurable $G$-action such that there is $d\geq 0$ and $p_0\geq 1$ with
\begin{align}\label{Eq:square invertible-Harich Chandra function-Discretization}
 \frac{\Xi_{X}}{(1+\fL)^{d}}\in L^{p_0+}(G).
\end{align}
Then 
\begin{align*}
 h_\eta(\Pi_{\eta},\nu_{\infty})\leq p_0\liminf_{n\rightarrow \infty}\dfrac{-1}{n}\int_{G}\log \Xi_X(s)\, d\eta^{*n}(s),
\end{align*}
where $ h_\eta(\Pi_{\eta},\nu_{\infty})$ is the Furstenberg entropy of the Poisson boundary.  If, in addition, $\xi$ is $\eta$-stationary, i.e. $(X,\xi)$ is a $(G,\eta)$-stationary space and $p_0=2$, then
\begin{align}\label{Eq:Lyap expo-weight coming from Harich Chandra function-inverse square integrable-Discretization}
 h_\eta(\Pi_{\eta},\nu_{\infty})=h_\eta(X,\xi)=-2\lim_{n\rightarrow \infty}\dfrac{1}{n}\int_{G}\log \Xi_X(s)\, d\eta^{*n}(s).
\end{align}
In particular, $(X,\xi)$ is the measure-preserving extension of the Poisson boundary 
$h_\eta(\Pi_{\eta},\nu_{\infty})$ as a $G$-space and the Harish-Chandra function $\Xi_{\Pi_\eta}$
satisfies \eqref{Eq:square invertible-Harich Chandra function-Discretization} with $p_0=2$. 
\end{thm}

\begin{proof}
Let $p>p_0$, and let
\begin{align*}
  \sg:=\frac{(1+\fL)^{d}}{\Xi_{X}}.
\end{align*}
It follows from Theorem \ref{T:Lyap Expo dominates entropy-inverse integrable weights} that 
$$ h_\eta(\Pi_{\eta},\nu_{\infty})\leq \liminf_{n\to \infty} \frac{1}{n} \int_{G} \log \sg^p (s)d\eta^{*n}(s).$$
On the other hand, by Theorem \ref{T:vanishing Lyap expo-log of length function}, 
$$\Lc_{1+\fL}(G,\eta)=\lim_{n\to \infty} \dfrac{1}{n}\int_{G}\log (1+\fL(s)) d\eta^{*n}(s)=0.$$
Hence, 
\begin{align*} 
  h_\eta(\Pi_{\eta},\nu_{\infty}) & \leq   \liminf_{n\to \infty} \frac{1}{n} \int_{G} \log \sg^p (s)d\eta^{*n}(s) \\ 
  & = p \liminf_{n\to \infty} \left[\dfrac{1}{n}\int_{G}\log (1+\fL(s))^d d\eta^{*n}(s)-  \dfrac{1}{n}\int_{G}\log \Xi_X(s)\, d\eta^{*n}(s) \right] \\ 
  & = p \liminf_{n\to \infty} \left[\dfrac{1}{n}\int_{G}(-\log \Xi_X)(s)\, d\eta^{*n}(s) \right].   
  \end{align*}
The result now follows by letting $p\to p_0^+$.

Now suppose that $\xi$ is $\eta$-stationary and $p_0=2$. Then we also have
\begin{align*}
  h_\eta(\Pi_{\eta},\nu_{\infty}) & \leq 2 \liminf_{n\to \infty} \left[\dfrac{1}{n}\int_{G}(-\log \Xi_X)(s)\, d\eta^{*n}(s) \right]  \\ 
  & \leq 2 \limsup_{n\to \infty} \left[\dfrac{1}{n}\int_{G}(-\log \Xi_X)(s)\, d\eta^{*n}(s) \right]  \\ 
  & \leq 2 \limsup_{n\to \infty} \left[\dfrac{1}{n}\int_{G}\int_X (-\log) \big[\rho(s,x)^{1/2}\big] d\xi(x) \, d\eta^{*n}(s) \right] \ \ (\text{as} -\log\ \text{is convex}) \\
    &=2\lim_{n\rightarrow \infty}\dfrac{1}{n}\int_{G}\int_X \bigg(-\frac{1}{2}\bigg)
    \log \rho(s,x) d\xi(x) d\eta^{*n}(s) \\
    &=\limsup_{n\to \infty} \dfrac{1}{n} h_{\eta^{*n}}(X,\xi) \\
    &=h_\eta(X,\xi) \\
    &\leq h_\eta(\Pi_{\eta},\nu_{\infty}). 
  \end{align*}
  That is, \eqref{Eq:Lyap expo-weight coming from Harich Chandra function-inverse square integrable-Discretization} holds. The last assertion follows from \cite[Theorem 4.4]{Furs-Glas 1} and the fact that Harish-Chandra functions of measure-preserving extensions are equal (see \cite[Proposition 1.8]{NevoZimmer}) (see also above).
\end{proof}

We now apply the preceding theorem to characterize the stationary actions of a group with property RD or the Kunze--Stein property under the assumption that the Koopman representation is weakly contained in the left-regular representation of the group. This generalizes \cite[Theorem 6.4]{ASdLSW1}, which dealt with countable discrete groups with rapid decay. 

\begin{thm}\label{T:charc amen action-RD groups}
Let $G$ be a second countable locally compact group, and let $\eta$ be a probability measure on $G$ with $\eta\in \BG$ such that $\eta$ has finite $n^{th}$-differential entropy, for all $n\in \N$.  Suppose that one of the following holds:\\
$(1)$ $(G,\fL)$ has property RD and $\eta$ has a finite $\log(1+\fL)$ moment.\\
$(2)$ $G$ has the Kunze-Stein Property.\\
Then 
for every $(G,\eta)$-stationary space $(X,\xi)$, the following are equivalent:
\begin{enumerate}[(i)]
    \item The space $(X,\xi)$ is a measure-preserving extension of the Poisson boundary of $(G,\eta)$.
    \item The action $G \curvearrowright (X,\xi)$ is amenable (in the sense of Zimmer).
    \item The representation $\pi_{X}$ is weakly contained in $\lambda_G$.
    \item $h_\eta(X,\xi)=h_\eta(\Pi_{\eta},\nu_{\infty})$.
\end{enumerate}
\end{thm}

\begin{proof}
(i) $\Rightarrow$ (ii): It is well known that groups acts amenably on their Poisson boundary \cite{Zimm 2}. Hence, the implication follows from the fact that measure-preserving extensions of amenable actions are amenable \cite[Corollary C]{Adam-Elliot-Gior} (see \cite[Theorem 2.4]{Zimm 2} for the special case of ergodic actions and \cite[Theorem 9.2]{nevosageev} for a characterization of the Poisson boundary as the unique minimal amenable $(G,\eta)$-stationary space).

(ii) $\Rightarrow$ (iii): This is a general consequence of the action being amenable \cite[Theorem 4.3.1]{Ana-Dela 1}.

(iii) $\Rightarrow$ (iv): This follows directly from Theorem \ref{T:Square integrable-inverse Harich Chandra function-the general case}.

(iv) $\Rightarrow$ (i): This follows from \cite[Theorem 4.4]{Furs-Glas 1}.
\end{proof}

Note that the above theorem may not hold for groups without property RD or Kunze-Stein. For instance, for amenable groups $G$, assertion (iii) always holds, whereas assertion (iv) is not automatic.

The following is an immediate corollary of the preceding theorem, which can be viewed as a ``representation theoretic version'' of the approach taken in \cite{Kaim 1}. The argument is also similar to the one given in \cite[Corollary 6.5]{ASdLSW1} so we omit the proofs. 

\begin{cor} \label{C:charc entopry-RD groups}
Let $G$ be a second countable locally compact group, and let $\eta$ be a probability measure on $G$ with $\eta\in \BG$ such that $\eta$ has finite $n^{th}$-differential entropy, for all $n\in \N$.  Suppose that one of the following holds:\\
$(1)$ $(G,\fL)$ has property RD and $\eta$ has a finite $\log(1+\fL)$ moment.\\
$(2)$ $G$ has the Kunze-Stein Property.\\
Let $(X,\xi)$ be a $\eta$-boundary. Then the following are equivalent:
\begin{enumerate}[(i)]
    \item The action $G \curvearrowright (X,\xi)$ is amenable (in the sense of Zimmer).
    \item The representation $\pi_{X}$ is weakly contained in $\lambda_G$.
    \item The space $(X,\xi)$ is the Poisson boundary of $(G,\eta)$.
\end{enumerate}
\end{cor}

\section{Discretization in groups with property RD or property KS}


In this section, we extend the main results established in Section \ref{sec:poissonboundariesstationarydynamicalsystems} and investigate the lattices associated with the class of groups considered therein. We begin with the following proposition, which generalizes \cite[Lemma 3.8]{Boyer 1}. Its proof follows arguments similar to those used in the proofs of \cite[Lemmas 3.7 and 3.8]{Boyer 1}.


\begin{lem} \label{lem:Square integrable-inverse Harich Chandra function-the general case}
Let $G$ be a locally compact second countable group with length function $\fL$, and let $\Gamma$ be a discrete subgroup in $G$. Let $\eta \in \B_c(G)$, and let $(X,\xi)$ be a $(G,\eta)$-stationary space. If
\begin{align*} 
 \frac{\Xi_{X}}{(1+\fL)^{d}}\in L^{p}(G),
\end{align*}
for some $p \geq 1$ and $d \geq 0$, then the restriction to $\Gamma$ satisfies
\begin{align*} 
 \frac{\Xi_{X}}{(1+\fL)^{d}}\in \ell^{p}(\Gamma).
\end{align*}
\end{lem}



\begin{proof}

Since $\Gamma$ is discrete subgroup of $G$, there is a compact neighborhood $V$ of the identity $e\in G$ such that $\{\gamma V: \gamma\in \Gamma \}$ are pairwise distinct. In particular, we must have 
\begin{align}\label{Eq:squre sum non dis vs dis-I}
   \sum_{\gamma\in \Gamma}  \int_V \frac{\Xi_{X}(\gamma s)^p}{(1+\fL(\gamma s))^{pd}} ds\leq  \int_G \frac{\Xi_{X}(s)^p}{(1+\fL(s))^{pd}} ds.
\end{align}
On the other hand, as $\fL$ is a length function and $V$ is compact, 
by standard arguments, there exist $C_1>0$ such that 
\begin{align}\label{Eq:squre sum non dis vs dis-II}
1+\fL(\gamma s) \leq C_1(1+\fL(\gamma)), \quad s\in V, \gamma\in \Gamma.
 \end{align}
Lastly, since $\eta$ has density with compact support, it is well-know that the Radon-Nikodym derivatives $\rho(\cdot,x)$ (associated with $(X,\xi)$ as a $(G,\eta)$-stationary space) are locally bounded 
(see \cite{Re-Bjr}). Hence, by \eqref{Eq:Radon-Nikodym cocycle-cocycle relations} and the fact that $V$ is compact, there exist $C_2>0$ such that for all $s\in V$ and $\xi$-almost every $x\in X$, we have
\begin{align} \label{Eq:Applicaton Hanckel inequality}
\rho(\gamma , x) \leq C_2(\rho(s\gamma , x)). 
\end{align}
This implies that
\begin{align}\label{Eq:squre sum non dis vs dis-III}
 \Xi_{X}(\gamma)\leq C_2\Xi_{X}(\gamma s), \quad s\in V, \gamma\in \Gamma.
\end{align}
The final results now follow by combining \eqref{Eq:squre sum non dis vs dis-I}, \eqref{Eq:squre sum non dis vs dis-II} , and \eqref{Eq:squre sum non dis vs dis-III}. 
\end{proof}

The following is our first main result.




\begin{thm}\label{T:Square integrable-inverse Harich Chandra function-Discretization} 
Let $\Gamma$ be a lattice subgroup in a locally compact second countable group $G$, let $\eta \in \B_c(G)$, and let $\mu$ be a Furstenberg discretization of $\eta$ (to $\Gamma$). Suppose that one of the following two conditions is satisfied:
\begin{enumerate}[(i)]
    \item the group $G$ has property RD with respect to a length function $\fL$ and $\mu$ has both finite entropy and finite $\log (1+\fL)$-moment;
    \item the group $G$ has the Kunze-Stein property and $\mu$ has finite entropy. 
\end{enumerate}
If $(\Pi_\eta, \nu_\infty)$ is the Poisson boundary of $(G,\eta)$, then
\begin{align}\label{Eq:maximal entropy-lattice}
h_\mu(\Pi_\eta, \nu_\infty)=h(\Gamma,\mu).
\end{align}
Consequently, $(\Pi_\eta, \nu_\infty)$, regarded as a $(\Gamma,\mu)$-stationary space, is the measure-preserving extension of the Poisson boundary of $(\Gamma,\mu)$. Moreover, if $\Gamma$ is cocompact, then $(\Pi_\eta, \nu_\infty)$ is isomorphic, as a $(\Gamma,\mu)$-stationary space, to the Poisson boundary of $(\Gamma,\mu)$.
\end{thm}

\begin{proof}
Let $(\Pi_\eta, \nu_\infty)$ be the Poisson boundary of $(G,\eta)$. Since $G\curvearrowright (\Pi_\eta,\nu_\infty)$ is amenable, $\pi_{\Pi_\eta}$ is weakly contained in $\lambda_G$ (by \cite[Theorem 4.3.1]{Ana-Dela 1}). Hence it follows that under either conditions (i) or (ii), we have 
\begin{align*}
 \frac{\Xi_{\Pi_\eta}}{(1+\fL)^{d}}\in L^{2+}(G),
\end{align*}
for some $d\geq 0$. 
This, together with Lemma \ref{lem:Square integrable-inverse Harich Chandra function-the general case}, implies that the restriction of
\begin{align*}
 \frac{\Xi_{\Pi_\eta}}{(1+\fL)^{d}} \ \text{on}\ G\ \text{belongs to}\  \ell^{2+}(\Gamma).
\end{align*}
Therefore, by Theorem \ref{T:Square integrable-inverse Harich Chandra function-the general case},
$h_\mu(\Pi_\eta, \nu_\infty)=h(G,\mu)$ (recall that $(\Pi_\eta, \nu_\infty)$ is also a $(G,\mu)$-stationary space as $\mu$ is a Furstenberg discretization of $\eta$). In particular, by \cite[Theorem 4.4]{Furs-Glas 1},  $(\Pi_\eta, \nu_\infty)$ is a measure-preserving extension of the Poisson boundary of $(\Gamma,\mu)$ as a $(\Gamma,\mu)$-stationary space. 


Now suppose that $\Gamma$ is cocompact. We show that in this case $(\Pi_\eta, \nu_\infty)$ is a $\mu$-boundary. Since, by \eqref{Eq:maximal entropy-lattice}, it also has maximal entropy, it follows that $(\Pi_\eta, \nu_\infty)$ must be the Poisson boundary of $(\Gamma,\mu)$. Let $x\in \Pi_\eta$. Since $\Pi_\eta$ is $\eta$-proximal, there is a net $\{t_i\}_{i\in I}$ in $G$ such that 
$t_i\, \nu_\infty$ converges in the weak$^*$-topology to the Dirac measure $\delta_y$, for some $y\in \Pi_\eta$. Let $W$  
be a compact fundamnetal domain of $\Gamma$ in $G$ so that such that $\{\gamma W: \gamma\in \Gamma \}$ are pairwise disjoint and $G=\bigcup_{\gamma\in \Gamma}  \gamma W$. Thus, we may write 
$t_i=\gamma_i s_i$, where $\gamma_i\in \Gamma$ and $s_i\in W$ for $i\in I$. On the other hand, it follows from \eqref{Eq:Radon-Nikodym cocycle-defn} and \eqref{Eq:Applicaton Hanckel inequality} (using that $W$ is compact) that there exist $C,D>0$ such that for every $\gamma\in \Gamma$, $s\in W$, and $\eta$-almost every $x\in X$, we have 
\begin{align*}
    C \frac{d(\gamma^{-1}\, \nu_\infty)}{d \nu_\infty}(x)\leq \frac{d(s^{-1}\gamma^{-1}\, \nu_\infty)}{d \nu_\infty}(x) \leq D \frac{d(\gamma^{-1}\, \nu_\infty)}{d \nu_\infty}(x).
\end{align*}
This, together with the fact that $t_i\, \nu_\infty$ converges to $\delta_y$ in the weak$^*$-topology, implies that $\gamma_i\, \nu_\infty$ converges to $\delta_y$ in the weak$^*$-topology. That is, $\Pi_\eta$ is $\mu$-proximal, so that $(\Pi_\eta, \nu_\infty)$ is a $\mu$-boundary. This completes the proof.  
\end{proof}

In the case where the lattice is irreducible, we can use a result of Bader and Shalom from \cite{badershalom} to show that the equality of the Poisson boundary of the lattice and its ambient group holds without the assumption of cocompactness.     

\begin{cor}
Let $G_1$ and $G_2$ be locally compact second countable groups, and let $\eta_1 \in \B_c(G_1)$ and $\eta_2 \in \B_c(G_2)$. Let $\Gamma$ be an irreducible lattice in $G:=G_1 \times G_2$, and let $\mu$ be a Furstenburg discretization of $\eta:=\eta_1\times \eta_2$. 
Suppose that the Poisson boundary of $(G,\eta)$ has a unique $\eta$-stationary measure. Then the Poisson boundary of $(\Gamma,\mu)$ coincides with the Poisson boundary of $(G,\eta)$ as a measurable $\Gamma$-space. 
\end{cor}

\begin{proof}
Let $(\Pi_\eta, \nu_\eta)$ and $(\Pi_\mu, \nu_\mu)$ be the Poisson boudaries of $(G,\eta)$ and $(\Gamma,\mu)$, respectively. It follows from Theorem \ref{T:Square integrable-inverse Harich Chandra function-Discretization} and \cite[Theorem 1.7]{badershalom} that $(\Pi_\mu, \nu_\mu)$ is $\Gamma$-isomorphic (as a $\Gamma$-space) to a $\eta$-boundary $(X,\xi)$ of $(G,\eta)$. 
Since $X$ is a $\Gamma$-space, by a result of Furstenberg, there is a $\Gamma$-equivariant map $\Phi: X \to P(\Pi_\eta)$. Let $\widetilde{\Phi}: X \to P(\Pi_\eta)$ be given by 
$$\widetilde{\Phi}(x)=\int_{G/\Gamma} \delta_s\Phi(\delta_{s^{-1}} x) d(s\Gamma) \ \ \ (x\in X),$$
where $d(s\Gamma)$ is the $G$-invariant probability measure on $G/\Gamma$. It is straightforward to verify that $\widetilde{\Phi}$ is well defined and, in fact, it is a $G$-equivariant map. The rest of the arguments follow in the same way as the proof of \cite[Theorem 9.2]{nevosageev}. Indeed, since $\Pi_\eta$ is $\eta$-proximal, we have from \cite[Corollary IV.2.10]{Marg1} that the push-forward measure $\widetilde{\Phi}^*(\xi)$ takes values on $\Pi_\eta$, $\eta$-almost surely. Hence, $\widetilde{\Phi}^*(\xi)$ becomes a $\eta$-stationary measure on $\Pi_\eta$ so that it must be the same as $\nu_\eta$ as it is the unique $\eta$-stationary measure on $\Pi_\eta$. In other words, there is a $G$-factor from the $\eta$-boundary $(X,\eta)$ into $(\Pi_\eta,\nu_\eta)$ so that they must be $G$-equivariantly isomorphic. 

\end{proof}

Let $G$ be a connected non-compact semisimple Lie group with finite center, let $K$ be a maximal compact subgroup of $G$ and $P$ a minimal parabolic subgroup, so that $G=KP$. For every $\eta\in \BG$, there is a unique probability measure $\nu_\infty$ on $G/P$ such that $(G/P,\nu_\infty)$ is the Poisson boundary of $(G,\eta)$ (\cite{Furstenberg1963a}).
Furthermore, for any lattice $\Gamma$ of $G$, there is a Furstenberg discretization $\mu$ of $\eta$ to $\Gamma$ such that $(G/P,\nu_\infty)$ is the Poisson boundary of $(\Gamma,\mu)$ (\cite{Furstenberg1973}). In fact, any Furstenberg discretizations of $\eta$ to $\Gamma$ having a finite first moment holds this property.

Below, we present an analogous of Theorem \ref{T:charc amen action-RD groups} for Furstenberg discretizations measure on $\Gamma$ of $\eta$ in the case where the Poisson measure on $G/P$ is the unique $K$-invariant probability measure. These results are not covered by Theorem \ref{T:charc amen action-RD groups} because many of such lattices may not have RD or Kuntz-Stein.  
We recall that a Riemannian metric on a connected semisimple Lie group is a smooth inner product on the tangent spaces of the group that allows the measurement of lengths, angles, and curvature. 
Such a metric is typically chosen to be left-invariant, ensuring that the geometry respects the group structure. In particular, it provides a natural length function on a connected semisimple Lie group.


\begin{thm}\label{T:Square integrable-inverse Harich Chandra function-lattice in Lie groups}
Let $G$ be a connected semisimple Lie group, let $\Gamma$ be a lattice in $G$, and let $\eta\in \BG$ be such that the Poisson measure on $G/P$ is the unique $K$-invariant probability measure $\nu_\infty^K$. Suppose that $\mu$ is a Furstenberg discretization of $\eta$ to $\Gamma$ such that $\mu$ has both finite entropy and finite logarithmic moment with respect to the Riemannian metric $\fL$ on $G$. 
Then 
for every $(\Gamma,\mu)$-stationary space $(X,\xi)$, the following are equivalent:
\begin{enumerate}[(i)]
    \item The space $(X,\xi)$ is a measure-preserving extension of the Poisson boundary of $(\Gamma, \mu)$.
    \item The action $\Gamma \curvearrowright (X,\xi)$ is amenable (in the sense of Zimmer).
    \item The representation $\pi_{X}$ is weakly contained in $\lambda_\Gamma$.
    \item $h_\mu(X,\xi)=h_\mu(G/P,\nu^K_{\infty})$.
\end{enumerate}
\end{thm}

\begin{proof}
Similar to the proof of Theorem \ref{T:charc amen action-RD groups}, we only need to show that (iii) implies (iv).  
Let us first recall some background. Suppose that $\nu_\infty^K$ is the unique $K$-invariant probability measure $G/P$. Then, by our assumption and the statement above, $(G/P, \nu^K_\infty)$ is the Poisson boundary of $(G,\eta)$ which is also the Poisson boundary of
$(\Gamma,\mu)$. In particular,
\begin{align}\label{Eq:2}
h_\mu(G/P, \nu^K_\infty)=h(\Gamma,\mu).
\end{align}
We now consider the Harish-Chandra–Schwartz space associated with $\Gamma$:
\begin{align*}
  \ell^\infty(\Gamma,\sg)=\{ f:\Gamma \to \C \ : \  \sup_{\gamma\in \Gamma} |f(\gamma)|\sg(\gamma) <\infty \},
\end{align*}
where $\fL$ is the Riemannian metric on $G$ and 
\begin{align*}
  \sg:=\frac{(1+\fL)^{d}}{\Xi_{G/P}}, \textrm{ with some}\ d>0.
\end{align*}
By \cite[Theorem 2.1]{Boyer 2}, 
there exists $d>0$ sufficiently large such that $\ell^\infty(\Gamma,\sg) \subseteq C^*_r(\Gamma)$.
Now assume (iii) holds. Since $\pi_X$ is weakly contained in the left regular representation $\lambda_\Gamma$, we have a canonical bounded linear mapping
\begin{align*}
    \ell^\infty(\Gamma,\sg) \subseteq C^*_r(\Gamma) \mapsto C^*_{\pi_X}.
\end{align*}
Therefore, using duality, we obtain
\begin{align*}
\Xi_X:=\la \pi_X(\cdot) 1_X, 1_X \ra  \in   \ell^\infty(\Gamma,\sg)^*=\ell^1(\Gamma,\sg^{-1})^{**}.
\end{align*}
Since $\Xi_X$ is a function on $\Gamma$, we must have 
\begin{align*}
\Xi_X \in \ell^1(\Gamma,\sg^{-1}),
\end{align*}
or equivalently,
\begin{align*}
\frac{\Xi_{X\times G/P}}{(1+\fL)^{d}}=\frac{\Xi_X\Xi_{G/P}}{(1+\fL)^{d}}=\Xi_X\sg^{-1}\in \ell^1(\Gamma).
\end{align*}
Here the measure on $X\times G/P$ is the product measure $\xi\times \nu^K_\infty$ and the action of $\Gamma$ on $X\times G/P$ is given by
$$\gamma\cdot(x,y)=(\gamma x,\gamma y).$$
Therefore, by Theorem \ref{T:Square integrable-inverse Harich Chandra function-the general case}, 
\begin{align*}
 h(\Gamma,\mu) & \leq  \liminf_{n\to \infty} \left[\dfrac{-1}{n}\sum_{\gamma\in \Gamma}\mu^{\ast n}(\gamma) \log \Xi_{X\times G/P}(\gamma) \right] \\
 & = \liminf_{n\to \infty} \left[\dfrac{-1}{n}\sum_{\gamma\in \Gamma}\mu^{\ast n}(\gamma) \log \left(\Xi_X(\gamma) \Xi_{G/P}(\gamma) \right)\right]  \\
 & = \liminf_{n\to \infty} \left[\dfrac{-1}{n}\sum_{\gamma\in \Gamma}\mu^{\ast n}(\gamma) \left[\log \left(\Xi_X(\gamma)\right)+\log \left(\Xi_{G/P}(\gamma)\right) \right]\right] \\
    &\leq \frac{1}{2} h_\mu(X,\xi)+ \frac{1}{2} h_\mu(G/P,\nu^K_\infty) \\
     &= \frac{1}{2} h_\mu(X,\xi)+ \frac{1}{2}  h(\Gamma,\mu) \ \ \ (\text{by}\ \eqref{Eq:2}).
  \end{align*}
Hence
\begin{align*}
\frac{1}{2} h(\Gamma,\mu) \leq \frac{1}{2} h_\mu(X,\xi) \leq \frac{1}{2} h(\Gamma,\mu) .
  \end{align*}
  The proof is complete.
\end{proof}


We finish this section with the following theorem that proves a weaker version of the result presented in Theorem \ref{T:Avez entropy-p conv operators-RD} but for a larger class of groups. 

\begin{thm}
Let $G$ be a connected semisimple Lie group, let $\Gamma$ be a lattice in $G$, and let $\eta\in \BG$ be such that the Poisson measure on $G/P$ is the unique $K$-invariant probability measure. Suppose that $\mu$ is a Furstenberg discretization of $\eta$ to $\Gamma$ such that $\mu$ has both finite entropy and finite logarithmic moment with respect to the Riemannian metric $\fL$ on $G$. 
Then
\begin{align*}
c(\Gamma,\mu) \leq h(\Gamma,\mu)\leq 2 c(\Gamma,\mu), 
\end{align*}
where $c(\Gamma,\mu)=\lim_{p\rightarrow \infty}-p\log r_{\pf_q(\Gamma)}(\mu).$
\end{thm}

\begin{proof}
Suppose that $\nu_\infty^K$ is the the unique $K$-invariant probability measure $G/P$. Then, $(G/P, \nu^K_\infty)$ is the Poisson boundary of $(G,\eta)$ which is also the Poisson boundary of $(\Gamma,\mu)$. 
Let $\fL$ be the Riemannian metric on $G$, and let 
\begin{align*}
  \sg:=\frac{(1+\fL)^{d}}{\Xi_{G/P}} \ \ (d>0).
\end{align*}
Once again, by applying \cite[Theorem 2.1]{Boyer 2}, there is $d>0$ such that 
$\ell^\infty(\Gamma,\sg)$ is a Banach algebra with respect to the convolution product (which, in particular, implies that $\sg$ is a weight on $\Gamma$) and we have
\begin{align*}
  \ell^\infty(\Gamma,\sg) \subseteq C^*_r(\Gamma).
\end{align*}
Thus, by the complex interpolation, 
for $1/p+1/q=1$, we have
\begin{align*}
  \ell^q(\Gamma,\sg_p) \subseteq \pf_q(\Gamma).
\end{align*}
It now follows from a similar argument given in the proof of Theorem \ref{T:Avez entropy-p conv operators-RD} (see also \cite[Theorem 5.2]{ASdLSW1}) that
\begin{align*}
 h_\sg(\Gamma,\mu)\leq \displaystyle\lim_{p\rightarrow \infty}-p\log r_{\pf_q(\Gamma)}(\mu).
\end{align*}
Hence, we have
\begin{align*}
 \lim_{p\rightarrow \infty}-p\log r_{\pf_q(\Gamma)}(\mu) &\geq  h_\sg(\Gamma,\mu) \\
& = h(\Gamma,\mu)-\Lc_\sg(\Gamma,\mu) \ \ \ \ (\text{by}\ \eqref{Eq:weighted Avez entropy-difference entropy and Lyap expo}) \\
& = h(\Gamma,\mu)-d\Lc_{1+\fL}(\Gamma,\mu)+\lim_{n\to \infty} \dfrac{1}{n}\sum_{\gamma\in \Gamma}\mu^{\ast n}(\gamma) \log \Xi_{G/P}(\gamma) \\
&= h(\Gamma,\mu) +\lim_{n\to \infty} \dfrac{1}{n}\sum_{\gamma\in \Gamma}\mu^{\ast n}(\gamma) \log \Xi_{G/P}(\gamma) \ \ (\text{by Theorem}\ \ref{T:vanishing Lyap expo-log of length function}) 
\end{align*}
Finally, 
by Theorems \ref{T:Square integrable-inverse Harich Chandra function-the general case} and \ref{T:Square integrable-inverse Harich Chandra function-Discretization},
\begin{align*}
 \lim_{n\to \infty} \dfrac{1}{n}\sum_{\gamma\in \Gamma}\mu^{\ast n}(\gamma) \log \Xi_{G/P}(\gamma)=-\frac{1}{2} h_\mu(G/P, \nu^K_\infty)=-\frac{1}{2}  h(\Gamma,\mu)
\end{align*}
so that 
$$2\lim_{p\rightarrow \infty}-p\log r_{\pf_q(\Gamma)}(\mu) \geq h(\Gamma,\mu).$$
The other side of the inequality of the theorem follows from Theorem \ref{T:Avez entropy-p conv operators-I}.
\end{proof}

\section{Furstenberg discretizations with finite moments}


We now give the proof of the existence of a Furstenberg discretization on \emph{cocompact} lattices stated in Theorem \ref{T:Furstenberg discretizasion-cocompact Lattice-The short version} (see also Definition \ref{D:Furstenberg discretization}). This provides an alternative proof of the result given in \cite[Proposition VI.4.1]{Marg1} with the advantage that our approach allows us to construct, in a systematic way, Furstenberg discretization measures satisfying finite moments of any order. We start with the following, which is \cite[Chapter VI, Lemma (A), page 207]{Marg1} with a slight generalization.


\begin{lem}\label{L:Margulis extensions}
Let $G$ be a locally compact, noncompact group, and let $\psi$ be a non-negative continuous function with compact support 
such that $\int_G \psi(s)ds=1$, and $\psi(e)>0$, and the set $\{s\in G: \psi(s)>0, \, \psi(s^{-1})>0 \}$ generates the group $G$. For every $s\in G$, there exists $\varepsilon>0$ and $n\in \N$ such that
\begin{align}\label{Eq:Positive difference-Conv powers}
   \delta_s*\psi^{*n}> \varepsilon \psi \textrm{ on } \supp\, \psi .
\end{align}
Suppose that $G$ is, moreover, noncompact. Then for every $\varepsilon_0 > 0$,
\begin{align}\label{Eq:Positive difference-Conv powers-boundedness}
\max \{n\in \N: \delta_{s}*\psi^{*n} \geq \epsilon \psi \ \text{for some}\ \varepsilon\geq \varepsilon_0 \ \text{and}\ s\in G \}
\end{align}
is finite.


\end{lem}

\begin{proof}

The proof of \eqref{Eq:Positive difference-Conv powers} is given in \cite[Chapter VI, Lemma (A), page 207]{Marg1}. Now suppose that $G$ is not compact. For every $s\in G$, $\varepsilon>0$ and $n\in \N$ satisfying
$\delta_{s}*\psi^{*n} \geq \varepsilon \psi$, we have
$$\la \psi^{*n} , \psi \ra \geq \varepsilon \la \delta_{s^{-1}}*\psi , \psi \ra =\varepsilon \psi*\check{\psi}(s).$$
Hence 
$$\la \psi^{*n} , \psi \ra \geq \varepsilon \|\psi*\check{\psi}\|_\infty=\varepsilon \|\psi\|_2^2.$$
On the other hand, since $G$ is not compact and the support of $\psi$ generates $G$, then  $\text{w}^*-\lim_{m \to \infty} \psi^{*m}=0$, where the weak$^*$ limit is taken in $(M(G),C_0(G))$ (\cite[Theorem 2]{Muh}). In particular, 
$$\lim_{m\to \infty} \la \psi^{*n} , \psi \ra=0.$$ These two relations imply that if $\varepsilon$ is away from 0 with a fixed constant, then the set considered in \eqref{Eq:Positive difference-Conv powers-boundedness} cannot be unbounded.
\end{proof}

Next, we establish some auxiliary results that will serve as the foundation for the proof of our main theorem.

\begin{thm}\label{T:weight associated to a generating measure}
Let $G$ be a locally compact, noncompact group, and let $\psi$ be a non-negative continuous function with compact support such that $\int_G \psi(s)ds=1$, and $\psi(e)>0$, and the set $\{s\in G: \psi(s)>0, \psi(s^{-1})>0 \}$ generates the group $G$. Then there exists a function $\tau: G \to (0,1]$ such that $\tau^{-1}$ is a weight on $G$ with $\tau(e)=1$ and a proper locally bounded function $\fL:G\to \N$ with $\fL(e)=1$ such that 
for all $s\in G$,
\begin{align}\label{Eq:weight associated to compactly supp function}
  \delta_s*\psi^{*\fL(s)}\geq \tau(s)\psi.
\end{align}
\end{thm}

\begin{proof}
 Now for every $s\in G$, we define
\begin{align}\label{Eq:Definition tau}
    \tau(s)=\sup\{ \epsilon\in (0,1) : \text{there is}\ n\in \N\ \text{such that}\  \delta_s*\psi^{*n}> \epsilon \psi\ \text{on}\ \supp\, \psi \}.
\end{align}
By Lemma \ref{L:Margulis extensions}, $\tau$ is well-defined. Clearly $\tau(e)=1$. Now we claim that 
\begin{align}\label{Eq:submultipicative relation of tau}
    \tau(st)> \tau(s)\tau(t) \ \ (s,t\in G).
\end{align}
Let $s,t\in G$. Pick $\epsilon_1,\epsilon_2\in (0,1)$ and $m,n\in \N$ such that 
$$\delta_s*\psi^{*m}> \epsilon_1 \psi \ , \   \delta_t*\psi^{*n}> \epsilon_2 \psi\ \text{on}\ \ \supp\, \psi.$$
Then 
$$\delta_{st}*\psi^{*(n+m-1)}\geq \epsilon_2 \left(\delta_s*\psi^{*m}\right) > \epsilon_1 \epsilon_2 \psi \ \ \text{on}\ \supp\, \psi.$$
Thus, by the definition of $\tau$ in \eqref{Eq:Definition tau}, 
$$\tau(st)> \epsilon_1 \epsilon_2.$$
By taking the suprimum over all possible choices of $\epsilon_1$ and $\epsilon_2$, the relation \eqref{Eq:submultipicative relation of tau} is obtained. Next, we show that $\tau$ is lower-semicontinuous on $G$. Let $s\in G$ and $0<\alpha<1$. By \eqref{Eq:Definition tau}, there is $n\in \N$ such that
$$\delta_s*\psi^{*n}> \alpha \tau(s) \psi \ \ \text{on}\ \supp\, \psi. $$
This
implies that 
$\delta_s*\psi^{*n}-\alpha \tau(s) \psi$ is continuous and positive on the compact set $\supp\, \psi$. Hence 
$$M=\inf \{(\delta_s*\psi^{*n})(t)-\alpha \tau(s) \psi(t) : t\in \supp\, \psi \}>0.$$
Now if we pick the compact neighbourhood $V$ of $e\in G$ such that 
$$\|\delta_{s'}*\psi^{*n}-\delta_s*\psi^{*n}\|_\infty <M \ \ \ (s'\in sV),$$
then, for every $t\in \supp\, \psi$ and $s'\in sV$, we must have 
\begin{align*}
  (\delta_{s'}*\psi^{*n})(t)-\delta_{s}*\psi^{*n}(t) & > - M\\  
  & \geq -\left[(\delta_s*\psi^{*n})(t)-\alpha \tau(s) \psi(t)\right] \\
  &= \alpha \tau(s) \psi(t)-(\delta_s*\psi^{*n})(t).
\end{align*}
Hence 
\begin{align*}
  (\delta_{s'}*\psi^{*n})(t) > \alpha \tau(s) \psi(t) \ \ (t\in \supp\, \psi, s'\in sV).
\end{align*}
In other words, by \eqref{Eq:Definition tau},
\begin{align*}
    \tau(s') > \alpha \tau(s) \ \ \ (s'\in sV).
\end{align*}
That is $\tau$ is lower-semicontinuous on $G$ so that it is Borel measurable and locally bounded below (meaning that $\inf_{t\in K} \tau(t)>0$ for all compact subset $K$ of $G$). We may now put $\om:=\tau^{-1}$
to obtain our desired weight. Now define the function $\fL:G\to \N$ as 
\begin{align}
 \fL(s)=\inf\{n\in \N: \delta_{s}*\psi^{*n} \geq \tau(s)\psi \} \ \ \ (s\in G).
\end{align}
Let us first show that $\fL$ is well-defined. 
Pick $s\in G$ and fix $\epsilon_0\in (0,\tau(s))$. 
By \eqref{Eq:Positive difference-Conv powers-boundedness}, 
\begin{align*}
   \max \{n\in \N: \delta_{s}*\psi^{*n} \geq \epsilon \psi \ \text{for some}\ \epsilon\in [\epsilon_0,\tau(s)) \}
\end{align*}
is finite. Hence we can find a sequence $\{\epsilon_k\}_{k\in \N}$ in $(\epsilon_0,\tau(s))$ and $n\in \N$
such that $\lim_{n\to \infty} \epsilon_k=\tau(s)$ and for every $k\in \N$, 
$$ \delta_{s}*\psi^{*n} \geq \epsilon_k \psi.$$
By taking a limit as $k$ tends to $\infty$, we have 
$$ \delta_{s}*\psi^{*n} \geq \tau(s) \psi.$$
That is $\fL$ is well-defined. Furthermore, since $\tau$ is locally bounded (below), it follows again from \eqref{Eq:Positive difference-Conv powers-boundedness} that $\fL$ is locally bounded.

\end{proof}

\begin{prop}\label{P:Furstenberg discretizasion-cocompact Lattice}
Let $G$ be a locally compact, noncompact group, let $\Gamma$ be a cocompact lattice in $G$ with a fundamental domain $W$, let $\om$ be a weight on $G$, and let $\tau$ is the function constructed in Theorem \ref{T:weight associated to a generating measure} in \eqref{Eq:Definition tau}. For every $\mu_0\in P(\Gamma)$ with finite support, consider the measurable map 
\begin{align*}
G\to P(\Gamma) \ \ , \ \ s \mapsto  \mu_s,    
\end{align*} 
given by 
$$\mu_{\gamma s}=\delta_{\gamma}*\mu_0 \ \ \ (\gamma\in \Gamma, s\in W).$$
Then the operator  
\begin{align*}
    T:L^1(G)\to \ell^1(\Gamma) 
\end{align*}
defined by
\begin{align}\label{Eq:Defn operator-Gamma to G}
   T(f)(\gamma)=\int_G \mu_s(\gamma)f(s)\tau(\gamma^{-1}s) ds \ \ \ \ (f\in L^1(G), \gamma\in \Gamma)
\end{align}
restrict to a bounded operator from $L^1(G,\om)$ into $\ell^1(\Gamma,\om)$. Moreover, there are positive constants $C$ and $D$ such that, for every $f\in L^1(G,\om)^+$,
\begin{align}
    C\|f\|_{\ell^1(\Gamma,\om)}\leq \|T(f)\|_{L^1(G,\om)}\leq D\|f\|_{\ell^1(\Gamma,\om)}.
\end{align}
\end{prop}

\begin{proof}
Without the loss of generality, we may assume that $\om$ is submutiplicative. 
Since $0\leq \tau\leq 1$, it follows that $T$ is well-defined. Moreover, for every $f\in L^1(G)^+$, 
\begin{align*}
   \|T(f)\|_{\ell^1(\Gamma,\om)} & =\sum_{\gamma\in \Gamma} \int_G \mu_s(\gamma)f(s)\tau(\gamma^{-1}s)\om(\gamma) ds \\
   &\leq \sum_{\gamma_1\in \Gamma} \sum_{\gamma\in \Gamma} \int_{W} \mu_0(\gamma_1^{-1}\gamma)f(\gamma_1 s) \om(\gamma)ds  \ \ \ (\text{as}\ 0\leq  \tau \leq 1 )\\
    &= \sum_{\gamma_1\in \Gamma} \sum_{\gamma\in \Gamma} \int_{W} \mu_0(\gamma)f(\gamma_1 s)\om(\gamma_1\gamma)ds \\
     &\leq \sum_{\gamma\in \Gamma} \mu_0(\gamma)\om(\gamma)\sum_{\gamma_1\in \Gamma} \int_{W} f(\gamma_1 s)\om(\gamma_1 s)\om(s^{-1})  ds \ \ \ (\text{as}\ \om \ \text{is submultiplicative})\\
       &\leq D \int_{W} \sum_{\gamma_1\in \Gamma} f(\gamma_1 s)\om(\gamma_1 s) ds \\
       &=D \|f\|_{L^1(G,\om)},
\end{align*}
where
\begin{align*}
    D=\|\mu_0\|_{\ell^1(G,\om)}\sup \{\om(s^{-1}): s\in W\}.
\end{align*}
We note that $D$ is finite since $\mu_0$ has finite support, $W$ is compact and $\om$ is locally bounded. Similarly, since $\tau^{-1}$ and $\om$ are both multiplicative, we also have
\begin{align*}
   \|T(f)\|_{\ell^1(\Gamma,\om)} &= \sum_{\gamma_1\in \Gamma} \sum_{\gamma\in \Gamma} \int_{W} \mu_0(\gamma)f(\gamma_1 s)\tau(\gamma^{-1} s)\om(\gamma_1\gamma)ds \\
     &\geq \sum_{\gamma\in \Gamma} \frac{\mu_0(\gamma)\tau(\gamma^{-1})}{\om(\gamma^{-1})}\sum_{\gamma_1\in \Gamma} \int_{W} f(\gamma_1 s)\om(\gamma_1 s)
     \tau(s)\om(s)^{-1} ds\\ 
      &\geq C \int_{W} \sum_{\gamma_1\in \Gamma} f(\gamma_1 s)\om(\gamma_1 s) ds \\
       &=C \|f\|_{L^1(G,\om)},
\end{align*}
where
\begin{align*}
    C=\|\mu_0 \check{\tau}\check{\om}^{-1}\|_{\ell^1(G)}\inf \{\tau(s)\om(s)^{-1}: s\in W\}.
\end{align*}
We note that $C$ is finite since $\mu_0$ has finite support, $W$ is compact and $\tau^{-1}$ and $\om$ are locally bounded (by Theorem \ref{T:weight associated to a generating measure}).
\end{proof}


We are now ready to prove our main result concerning Furstenberg discretization on cocompact lattices. We present and prove a more detailed version of the theorem than the statement given in Theorem \ref{T:Furstenberg discretizasion-cocompact Lattice-The short version}. In our view, this expanded formulation provides greater insight into the underlying ideas and techniques of the proof. Moreover, it highlights the key components of our approach, which we hope can be suitably adapted to address the non-cocompact setting in future work.


\begin{thm}\label{T:Furstenberg discretizasion-cocompact Lattice}
Let $G$ be a locally compact, noncompact group, let $\Gamma$ be a cocompact lattice in $G$, and let $\eta\in B_c(G)$.
Then:\\
$(i)$ there exists a probability measure $\mu$ on $\Gamma$  with $\supp\, \mu=\Gamma$ such that every $(G,\eta)$-stationary space is also a $(\Gamma,\mu)$-stationary space. \\
(ii) $\mu$ can be constructed having the following formula: 
\begin{align}\label{Eq:Furstenberg Discritization measure}
    \mu(\gamma)=\sum_{\gamma_1\in \Gamma} \int_{W} \mu_0(\gamma_1^{-1}\gamma)\tau(\gamma^{-1}\gamma_1s) \phi(\gamma_1s) ds \ \ \ \ (\gamma\in \Gamma). 
\end{align}
Here $\mu_0$ is any finitely supported probability measure on $\Gamma$, $\tau$ is a function of a type constructed in Theorem \ref{T:weight associated to a generating measure} in \eqref{Eq:Definition tau}, $W$ is an arbitrary fundamental domain of $G$ with respect to $\Gamma$, and $\phi$ is continuous, non-negative function on $G$ (depending on $\eta$, $\tau$ and $\mu_0$) with $\phi\in L^1(G)$. \\
$(iii)$ 
The measure
$\mu$ constructed in \eqref{Eq:Furstenberg Discritization measure} has finite $m$-moment, for every $m\in \N$, with respect to any subadditive function $\sg$ on $G$ (after it is restricted to $\Gamma$). 
\end{thm}

\begin{proof}

We first note that any $\eta$-stationary measure is also $\eta^{*k}$-stationary for every $k\in \N$. Hence any Furstenberg discretization $\mu$ of $\eta^{*k}$ is also a Furstenberg discretization of $\eta$. Also if we let $k_0\in \N$ to be large enough, then the density function $\psi$ of $\eta^{*k_0}$ (i.e. $\psi=\frac{d(\eta^{*k_0})}{ds}$) is continuous and compactly supported on $G$, and also satisfies the assumption of Theorem \ref{T:weight associated to a generating measure} (this is because $\eta$ is compactly supported with bounded density and spread out). Thus it suffices to construct $\mu$ as a discretization of $\eta^{*k_0}$. Accordingly, for the rest of the proof, we let $\psi$ to be as above and $\tau$ to be the function constructed in \eqref{Eq:Definition tau} with respect to $\psi$.


Let us fix a measurable map 
\begin{align}
G\to P(\Gamma) \ \ , \ \ s \mapsto  \mu_s.    
\end{align}  
We now define two bounded operators 
\begin{align*}
    T:L^1(G)\to \ell^1(\Gamma) \ , \ S:L^1(G)\mapsto L^1(G)
\end{align*}
as follows:
\begin{align}\label{Eq:Defn operator-Gamma to G}
   T(f)(\gamma)=\int_G \mu_s(\gamma)f(s)\tau(\gamma^{-1}s) ds \ \ \ \ (f\in L^1(G))
\end{align}
and
\begin{align}\label{Eq:Defn operator-G to Gamma}
   S(f)=\int_G\sum_{\gamma\in \Gamma} \mu_s(\gamma)  \left(\delta_s*\psi^{*\fL(\gamma^{-1}s)}-\tau(\gamma^{-1}s) \delta_\gamma*\psi \right) f(s)ds  \ \ \ \ \ (f\in L^1(G)) .
\end{align}
Here $\tau$ and $\fL$ are the functions constructed in Theorem \ref{T:weight associated to a generating measure}. Since $0\leq \tau\leq 1$, $\psi$ is continuous and non-negative with compact support, and $\int_G \psi(s)ds=1$, it follows that both $T$ and $S$ are well defined. Also, by \eqref{Eq:weight associated to compactly supp function} (applied to $S$),
\begin{align}\label{Eq:Defn operator-positive to positive}
    T(L^1(G)^+)\subseteq \ell^1(\Gamma)^+ \ , \ S(L^1(G)^+)\subseteq L^1(G)^+. 
\end{align}
We now build a sequence $\{\psi_n\}_{n\in \N}$ in $L^1(G)^+\cap C(G)$ as follows:
\begin{align}
   \psi_1:=\psi \ , \ \psi_{n+1}=S(\psi_n) \ \ (n\in \N). 
\end{align}
It follows from \eqref{Eq:Defn operator-G to Gamma} and \eqref{Eq:Defn operator-positive to positive} that
\begin{align}\label{Eq:inductive construction-prob measure}
    \psi_{n+1}=\int_G\sum_{\gamma\in \Gamma} \mu_s(\gamma)  \left(\delta_s*\psi^{*\fL(\gamma^{-1}s)}\right)\psi_n(s)ds -T(\psi_n)*\psi.  
\end{align}
Hence
\begin{align*}
  \|\psi_{n+1}\|_{L^1(G)}& =\int_G \psi_{n+1}(t)dt \\
  & =\int_G \sum_{\gamma\in \Gamma} \mu_s(\gamma) \int_G \left(\delta_s*\psi^{*\fL(\gamma^{-1}s)}\right)(t)dt \psi_n(s)ds -\int_G \left(T(\psi_n)*\psi\right)(t) dt \\
  &= \int_G\sum_{\gamma\in \Gamma} \mu_s(\gamma)  \psi_n(s)ds -\sum_{\gamma\in \Gamma} T(\psi_n)(\gamma) \ \ \ \left(\text{as}\ \int_G \psi(t)dt=1 \right) \\
&= \int_G \psi_n(s)ds -\sum_{\gamma\in \Gamma} T(\psi_n)(\gamma) \ \ \ \ (\text{as}\ \ \mu_s\in P(\Gamma)) \\
&=  \|\psi_{n}\|_{L^1(G)} -\sum_{\gamma\in \Gamma} T(\psi_n)(\gamma). 
\end{align*}
Therefore, $\{\|\psi_{n}\|_{L^1(G)}\}_{n\in \N}$ is decreasing and we have
\begin{align*}
   \left\|\sum_{n=1}^\infty T(\psi_n)\right\|_{\ell^1(\Gamma)} &=\sum_{n=1}^\infty \sum_{\gamma\in \Gamma} T(\psi_n)(\gamma)   \\
   & = \sum_{n=1}^\infty \left(\|\psi_{n}\|_{L^1(G)}-\|\psi_{n+1}\|_{L^1(G)}\right) \\   
   &=1-\lim_{n\to \infty} \|\psi_n\|_{L^1(G)}.
\end{align*}
In other words, if we let
\begin{align}\label{Eq:Furstenberg disc-construction-I}
   \mu:=\sum_{n=1}^\infty T(\psi_n) \ \ \text{and}\ \ \phi:=\sum_{n=1}^\infty \psi_n, 
\end{align}
then, by also applying Proposition \ref{P:Furstenberg discretizasion-cocompact Lattice},
\begin{align}\label{Eq:Furstenberg disc-construction-II}
 \mu\in \ell^1(\Gamma)^+ , \phi\in L^1(G), \ \textrm{and} \ \ \|\mu\|_{\ell^1(\Gamma)}= 1-\lim_{n\to \infty} \|\psi_n\|_{L^1(G)}=1.
\end{align}
In particular, $\mu$ is a probability measure on $\Gamma$ so that (ii) holds.

Now suppose that $(X,\xi)$ is a $(G,\mu)$-stationary space. If we consider positive measures
\begin{align*}
  \forall n\in \N: \eta_n=\psi_nds,   
\end{align*}
then, by \eqref{Eq:inductive construction-prob measure} and the fact that $\xi$ is $\eta^{*k_0}$-stationary and $d\eta^{*k_0}=\psi ds$, we obtain
\begin{align*}
    \psi_{n+1}*\xi=\psi_n*\xi-T(\psi_n)*\xi. 
\end{align*}
Hence
\begin{align}\label{Eq:Furstenberg disc-construction-III}
    \mu*\xi=\xi-\lim_{n\to \infty} \psi_n*\xi=\xi,
\end{align}
as $\lim_{n\to \infty} \|\psi_n\|_{L^1(G)}=0$. Thus $\mu$ is our desired measure, i.e. $\xi$ is $\mu$-stationary so that (i) holds.


(iii) Let us first fix some notations. For any weight $\om$ on $G$, let
 \begin{align*}
   \fW_\om:=\sup \left\{\om(t)^{\fL(\gamma^{-1}s)} : t\in \supp\, \psi, \gamma\in \supp\, \mu_0, s\in W \right \}.
\end{align*}
We note that $\fW_\om$ is finite since $W$ and $\supp\, \psi$ are compact, $\supp\, \mu_0$ is finite, and $\om$ and $\fL$ are locally bounded.
Let $s\in G$ and write $s=\gamma_1s_1$ with $\gamma_1\in \Gamma$ and $s_1\in W$. Then we have 
   \begin{align*}
  \int_G  \sum_{\gamma\in \Gamma} \mu_s(\gamma) \left(\psi^{*\fL(\gamma^{-1}s)}\right)(t)\om(t)dt  &= \int_G  \sum_{\gamma\in \Gamma}  \mu_0(\gamma_1^{-1}\gamma) \left(\psi^{*\fL(\gamma^{-1}\gamma_1 s_1)}\right)(t) \om(t) dt  \\
       &= \int_G \sum_{\gamma\in \Gamma} \mu_0(\gamma)  \left(\psi^{*\fL(\gamma^{-1}s_1)}\right)(t)\om(t) dt\\
        & \leq \fW_\om. 
       \end{align*}
       Hence 
         \begin{align}\label{Eq:upper bounded-convolution of prob mesures}
  \sup \left\{\int_G \sum_{\gamma\in \Gamma} \mu_s(\gamma)\left(\psi^{*\fL(\gamma^{-1}s)}\right)(t)\om(t)dt : s\in G \right\}\leq \fW_\om.
\end{align}
Now let $\sg$ be a subadditive function on $G$. Since $G$ is compactly generated, any subadditive function is dominated by a multiple of the length function coming from a fixed generating compact symmetric subset of $G$. Hence it suffices to prove our claim when $\sg=1+l_U$, where $U$ is a (fixed) generating compact symmetric subset of $G$ and $l_U$ is its corresponding length function. To prove the claim, we apply induction on $m\in \N\cup \{0\}$ to show that for
\begin{align*}
    \phi=\sum_{n=1}^\infty \psi_{n} \ \ \text{and}\ \ \mu=T(\phi), 
\end{align*}
we have that
\begin{align}\label{Eq:m-moment condition-cocompact lattice}
    \phi\in L^1(G,\sg^m) \ \ \text{and} \ \ \mu\in \ell^1(G,\sg^m).
\end{align}
The result holds for $m=0$ by part (ii). Now suppose that it is true for every $k=0,\ldots, m-1$ for some $m\geq 1$. We have
\begin{align*}
& \int_G \int_G \sum_{\gamma\in \Gamma} \mu_s(\gamma) \left(\delta_s*\psi^{*\fL(\gamma^{-1}s)}\right)(t)\sg(t)^m \psi_n(s) dt ds \\ 
  &=\int_G \int_G \sum_{\gamma\in \Gamma} \mu_s(\gamma) \left(\psi^{*\fL(\gamma^{-1}s)}\right)(t)\sg(st)^m\psi_n(s) dt ds \\
   &\leq \int_G \int_G \sum_{\gamma\in \Gamma} \mu_s(\gamma) \left(\psi^{*\fL(\gamma^{-1}s)}\right)(t)(\sg(s)+\sg(t))^m\psi_n(s) dt ds \\
    &=\sum_{k=0}^m a_k \int_G  \int_G \sum_{\gamma\in \Gamma} \mu_s(\gamma)\left(\psi^{*\fL(\gamma^{-1}s)}\right)(t)\sg(t)^k\sg(s)^{m-k}\psi_n(s) dt ds \\
     &\leq \int_G \psi_n(s) \sg(s)^m ds 
      +\sum_{k=1}^m a_k \fW_{\sg^k} \int_G \sg(s)^{m-k}\psi_n(s) ds \ \ \ (\text{by}\ \eqref{Eq:upper bounded-convolution of prob mesures}) \\
      &= \|\psi_{n}\|_{L^1(G,\sg^m)} 
      +\sum_{k=1}^m a_k \fW_\sg^k \|\psi_{n}\|_{L^1(G,\sg^{m-k})}.
\end{align*}
By summarizing what we have above, we obtain  
\begin{align*}
  \|\psi_{n+1}\|_{L^1(G,\sg^m)}  +\|T(\psi_n)*\psi\|_{L^1(G,\sg^m)} & =\int_G \psi_{n+1}(t)\sg(t)^mdt +\int_G \left(T(\psi_n)*\psi\right)(t) \sg(t)^m dt\\
  &= \int_G  \sum_{\gamma\in \Gamma} \mu_s(\gamma) \left(\delta_s*\psi^{*\fL(\gamma^{-1}s)}\right)(t)\sg(t)^m\psi_n(s) dt ds \\
    & \leq \|\psi_{n}\|_{L^1(G,\sg^m)} 
      +\sum_{k=1}^m a_k \fW_\sg^k \|\psi_{n}\|_{L^1(G,\sg^{m-k})}.
\end{align*}
Hence, by induction on $n$, each $\|\psi_{n}\|_{L^1(G,\sg^m)}$ is finite and we have
\begin{align*}
\|T(\psi_n)*\psi\|_{L^1(G,\sg^m)} \leq  \sum_{k=1}^m a_k \fW_\sg^k \|\psi_{n}\|_{L^1(G,\sg^{m-k})} +\|\psi_{n}\|_{L^1(G,\sg^m)}-\|\psi_{n+1}\|_{L^1(G,\sg^m)}. 
\end{align*}
Thus, since $\phi:=\sum_{n=1}^\infty \psi_{n}$ and $\mu=\sum_{n=1}^\infty T(\psi_n)$,
\begin{align*}
\|\mu*\psi \|_{L^1(G,\sg^m)} &= \sum_{n=1}^\infty \|T(\psi_n)*\psi\|_{L^1(G,\sg^m)} \\
& \leq \sum_{n=1}^\infty  \left[\sum_{k=1}^m a_k \fW_\sg^k \|\psi_{n}\|_{L^1(G,\sg^{m-k})}+\|\psi_{n}\|_{L^1(G,\sg^m)}-\|\psi_{n+1}\|_{L^1(G,\sg^m)} \right] \\
&\leq \sum_{k=1}^m a_k \fW_\sg^k \|\phi\|_{L^1(G,\sg^{m-k})}+\|\psi\|_{L^1(G,\sg^m)}. 
\end{align*}
By our induction assumption, the quantity on the right-hand side of the above formula is finite so that we must have
$\mu*\psi\in L^1(G,\sg^m)$. However, 
since for every $s,t\in G$, 
$$\sg(st)^m\leq (\sg(s)+\sg(t))^m\leq 2^m \sg(s)^m \sg(t)^m,$$
we get
\begin{align*}
2^m\int_G \left(\mu*\psi\right)(t)\sg^m(t) dt  &= 2^m\int_G\sum_{\gamma\in \Gamma} \mu(\gamma) \psi(t) \sg(\gamma t)^m dt \\
&\geq  \int_G\sum_{\gamma\in \Gamma} \mu(\gamma) \psi(t) \sg(\gamma)^m\sg(t^{-1})^{-m} dt \\
&= \|\mu\|_{\ell^1(\Gamma,\sg^m)}\int_G \psi(t) \sg(t^{-1})^{-m} dt. 
\end{align*}
Thus $\mu\in \ell^1(\Gamma, \sg^m)$ so that, by Proposition \ref{P:Furstenberg discretizasion-cocompact Lattice}, $\phi\in L^1(G,\sg^m).$ Thus our induction assumption also holds for $m$. This completes the proof.
\end{proof}

\end{document}